\documentclass[12pt,a4paper]{amsart}
\usepackage[T1]{fontenc}
\usepackage[top=35mm, bottom=35mm, left=30mm, right=30mm]{geometry}

\usepackage{xcolor}
\usepackage{amsfonts} 

\usepackage[looser]{newtxtext}
\usepackage[smallerops]{newtxmath}
\usepackage{verbatim}
\usepackage{enumerate}
\usepackage{fancyhdr}
\makeatletter
\def\@@and{\MakeLowercase{and}}
\makeatother

\usepackage[hidelinks]{hyperref}

\theoremstyle{definition}
\newtheorem{defn}{Definition}[section]
\newtheorem{exam}[defn]{Example}
\newtheorem{question}[defn]{Question}
\newtheorem{rem}[defn]{Remark}

\theoremstyle{plain}
\newtheorem{thm}[defn]{Theorem}
\newtheorem{lem}[defn]{Lemma}
\newtheorem{prop}[defn]{Proposition}
\newtheorem{coro}[defn]{Corollary}

\newcommand{\calpfg}{\mathcal{P}_f(G)}
\newcommand{\calpg}{\mathcal{P}(G)}
\newcommand{\ft}{\mathcal{F}_{\textup{t}}}
\newcommand{\fs}{\mathcal{F}_{\textup{s}}}
\newcommand{\fip}{\mathcal{F}_{\textup{ip}}}
\newcommand{\fps}{\mathcal{F}_{\textup{ps}}}
\newcommand{\finf}{\mathcal{F}_{\textup{inf}}}
\newcommand{\fpud}{\mathcal{F}_{\textup{pud}}}
\newcommand{\fpubd}{\mathcal{F}_{\textup{pubd}}}
\newcommand{\fcen}{\mathcal{F}_{\textup{cen}}}

\newcommand{\bbn}{\mathbb{N}}

\newcommand{\bbnz}{{\mathbb{N}_0}}
\newcommand{\calf}{\mathcal{F}}

\newcommand{\calb}{\mathcal{B}}
\DeclareMathOperator{\cl}{cl}

\allowdisplaybreaks %

\title[R\MakeLowercase{eturn time sets and product recurrence}] 
{R\MakeLowercase{eturn time sets and product recurrence}}

\author[J. L\MakeLowercase{i}]{J\MakeLowercase{ian} L\MakeLowercase{i}}
\address[J. Li]{Institute for  Mathematical Sciences and Artificial Intelligence \& Department of Mathematics,
	Shantou University, Shantou, 515821, Guangdong, China}
\email{lijian09@mail.ustc.edu.cn}
\urladdr{https://orcid.org/0000-0002-8724-3050}

\author[X. L\MakeLowercase{iang}]{X\MakeLowercase{ianjuan} L\MakeLowercase{iang}}
\address[X. Liang]{Department of Mathematics,
	Yunnan Normal University, Kunming, 650000, Yunnan, China}
\email{liangxianjuan@outlook.com}
\urladdr{https://orcid.org/0000-0003-1970-3809}

\author[Y. Y\MakeLowercase{ang}]{Y\MakeLowercase{ini} Y\MakeLowercase{ang}}
\address[Y. Yang]{Department of Mathematics,
	Shantou University, Shantou, 515821, Guangdong, China}
\email{ynyangchs@foxmail.com}
\urladdr{https://orcid.org/0000-0001-6564-2213}



\date{\today}
 
\begin{document}

\begin{abstract}
Let $G$ be a countable infinite discrete group.
We show that a subset $F$ of $G$ contains a return time set of some 
piecewise syndetic recurrent point $x$ in a compact Hausdorff space $X$ with a $G$-action if and only if $F$ is a quasi-central set.
As an application, we show that if a nonempty closed subsemigroup $S$ of the Stone-\v{C}ech compactification $\beta G$ contains the smallest ideal $K(\beta G)$ of $\beta G$ then $S$-product recurrence is equivalent to distality, 
which partially answers a question of Auslander and Furstenberg (Trans. Amer. Math. Soc. 343, 1994, 221--232).
\end{abstract}

\maketitle

\section{Introduction}

By a topological dynamical system, we mean a pair $(X,T)$,
where $X$ is a compact metric space with a metric $d$ and
$T\colon X\to X$ is a continuous map.
The study of recurrence is one of the central topics in topological dynamics.
For a point $x\in X$ and a subset $U$ of $X$, 
the return time set of $x$ to $U$ (In this paper, “neighborhood” always signifies an open neighborhood) is 
\[
N(x,U)=\{n\in\bbnz\colon T^n x\in U\},
\]
where $\bbn_0$ denote the collection of
non-negative integers.
Recurrent time sets are closely associated with 
the combinatorial property of the sets of non-negative integers. 
In the seminal monograph \cite{F81},  
Furstenberg characterized the return time sets of a recurrent point via IP-subsets of $\bbnz$ which is defined combinatorially.
Recall that a point $x\in X$ is called recurrent if for every neighborhood $U$ of $x$, the recurrent time set $N(x,U)$ is infinite, and a subset $F$ of $\bbnz$ is called an IP-set if there exists a sequence 
$\{p_i\}_{i=1}^\infty$ in $\bbnz$ such that the
finite sum $FS(\{p_i\}_{i=1}^\infty)$ of $\{p_i\}_{i=1}^\infty$ is infinite and contained in $F$, where
\[
FS(\{p_i\}_{i=1}^\infty)= \Bigl\{\sum_{i\in \alpha} p_i \colon \alpha
\text{ is a  nonempty finite  subset of }  \bbn \Bigr\}.
\]

\begin{thm}[{\cite[Theorem 2.17]{F81}}] \label{thm:Fur-recurrent-IP}\ 
	\begin{enumerate} 
		\item 
		Given a topological dynamical system $(X,T)$, if a point $x\in X$ is recurrent, 
		then for any neighborhood $U$ of $x$,  $N(x,U)$ is an IP-set.
		\item If a subset $F$ of $\bbnz$ is an IP-set, 
		then there exists a topological dynamical system $(X,T)$, a recurrent point $x\in X$ and a neighborhood $U$ of $x$
		such that $N(x,U)\subset F\cup\{0\}$.
	\end{enumerate}
\end{thm}

Furstenberg introduced the concept of central subsets of $\bbn_0$ and proved
the so-called "Central sets theorem" (see \cite[Proposition 8.21]{F81}), which has many combinatorial consequences. 
For a recent survey on central sets, we refer the reader to \cite{H20}. 
In \cite{HMS96} Hindman et al.\@ introduced the notion of quasi-central sets, and both concepts were further generalized to be applicable to arbitrary semigroups. 
Motivated by Theorem \ref{thm:Fur-recurrent-IP}, 
we characterize the return time sets of a piecewise syndetic recurrent point via quasi-central subsets of $\bbn_0$. 

\begin{thm}\label{thm:main1}
	\begin{enumerate}
		\item Given a topological dynamical system $(X,T)$, if a point $x\in X$ is piecewise syndetic recurrent, then for
		every neighborhood $U$ of $x$, $N(x,U)$ is a quasi-central set;
		\item For any quasi-central
		subset $F$ of $\bbn_0$, 
		there exists a topological system 
		$(X,T)$, a piecewise syndetic  recurrent point
		$x\in X$ and a neighborhood $U$ of $x$
		such that $N(x,U)\subset F\cup \{0\}$.
	\end{enumerate}
\end{thm}

The proof of Theorem \ref{thm:main1} is presented in  \ref{section:thm1.2}.
In fact, we will show that a more general version of Theorem~\ref{thm:main1} also holds for $G$-system and some special kinds for recurrence, see Theorem~\ref{thm:main3} for details. 
Recall that a $G$-system is a pair $(X,G)$, where $X$ is a compact Hausdorff space and $G$ is a countable discrete group continuously acting on $X$. 
A key aspect of the proof of Theorem~\ref{thm:main3} is a "purely" combinatorial characterization of the recurrent time sets corresponding to certain specific types of recurrent points, see Theorem \ref{thm:main2}. 

Let $(X,T)$ be a topological dynamical system.
Recall that two points $x,y\in X$ are called proximal if $\liminf_{k\to\infty}d(T^k x,T^k y)=0$,
and a point $x\in X$ is called distal if it is not proximal to any point in its orbit closure other than itself. 
By the well-known Auslander-Ellis theorem 
(see e.g.~\cite[Theorem 8.7]{F81}),
any distal point is uniformly recurrent.
In \cite{F81}, Furstenberg also characterized distal points 
in terms of  recurrent time sets and synchronized recurrence 
with certain types of recurrent points (see \cite{EEN01} and \cite{DLX21} for $G$-systems).
Recall that a subset $F$ of $\bbnz$ is called an IP$^*$-set if for any IP-subset $F'$ of $\bbnz$, $F\cap F'\neq\emptyset$.

\begin{thm}[{\cite[Theorem 9.11]{F81}}] \label{thm:Fur-distal-point}
	Let $(X,T)$ be a topological dynamical system and $x\in X$. 
	Then the following assertions are equivalent:
	\begin{enumerate}
		\item $x$ is distal;  
		\item $x$ is IP$^*$-recurrent, that is, for  any neighborhood $U$ of $x$,  $N(x,U)$ is an IP$^*$-set;
		\item $x$ is product recurrent, that is, for any topological dynamical system $(Y,S)$ and any recurrent point $y\in Y$, $(x,y)$ is recurrent in the product system $(X\times Y,T\times S)$;
		\item for any topological dynamical system $(Y,S)$ and any uniformly recurrent point $y\in Y$, $(x,y)$ is uniformly recurrent in the product system $(X\times Y,T\times S)$.
	\end{enumerate}
\end{thm}

In \cite{AF94}, Auslander and Furstenberg treated directly the action $E\times X\ni(p,x)\mapsto px\in X$ of a compact right topological semigroup $E$ on a compact Hausdorff space $X$. 
It should be noticed that the maps $x\mapsto px$ are often discontinuous for such semigroup actions.
Such an action is referred to as an Ellis action in \cite{AAG08}.
Within this framework the authors of \cite{AAG08} investigated the relationships between dynamics of an action and an algebraic structure of $E$. 
For instance, they obtained several characterizations of distal, semidistal and almost-distal flows for an Ellis action.
The Stone-\v{C}ech compactification $\beta G$ of a discrete group $G$ forms a compact right topological semigroup, and its action constitutes an important example of Ellis action (referred to as a $\beta G$-action). 

Partially motivated by Theorem~\ref{thm:Fur-distal-point}, Auslander and Furstenberg \cite{AF94} introduced the concept of $S$-product recurrence for a closed subsemigroup $S$ of $E$, and showed that under certain conditions, a point is $S$-product recurrent if and only if it is a distal point.
In the end of the paper \cite{AF94}, Auslander and Furstenberg asked the following two questions:

\begin{question}\label{question:AF94-1}
	How to characterize the closed subsemigroups $S$ of a compact right topological semigroup for which an $S$-product recurrent point is  distal?
\end{question}

\begin{question}\label{question:AF94-2}
	If $(x, y)$ is recurrent for any almost periodic point $y$, 
	is $x$ necessarily a distal point?
\end{question}

Question~\ref{question:AF94-2} was answered negatively by Haddad and Ott in \cite{HO08} for topological dynamical systems.
In fact, this question is related to dynamical systems which are disjoint from all minimal systems. 
In \cite{DSY12}, Dong, Shao and Ye studied general product recurrence properties systematically and in \cite{OZ13} Oprocha and Zhang showed that 
if $(x,y)$ is recurrent for any piecewise syndetic recurrent point $y$, then $x$ is a distal point.

Recall that the Stone-\v{C}ech compactification  $\beta G$ of $G$ has a smallest ideal $K(\beta G)$ which is the union of all minimal left ideals of $\beta G$. 
We consider $\beta G$-actions on compact Hausdorff spaces and obtain the following sufficient condition for the closed subsemigroups $S$ of $\beta G$ for which an $S$-product recurrent point is a distal point, partly answering Auslander and Furstenberg's Question~\ref{question:AF94-1}. 

\begin{thm}\label{thm:main-result-S-product-recurrent}
	Let $(X,\beta G)$ be a $\beta G$-action
	and $S$ be a nonempty closed subsemigroup of $\beta G\setminus G$. 
	If $K(\beta G)\subset S$, then
	a point $x\in X$ is distal if and only if $x$ is $S$-product recurrent.
\end{thm}

As an application, we obtain a characterization of distal points in terms of product recurrence for $G$-systems on compact Hausdorff spaces. It should be noted that some special cases for a topological dynamical system $(X,T)$ were established by Oprocha and Zhang  in \cite{OZ13}. 

\begin{thm}\label{thm:main-result-distal-product-rec}
	Let $G$ be a countable infinite discrete group and 
	$\calf\subset \calpg$ be a Furstenberg family. 
	If $\calf$ has the Ramsey property and the hull of $\calf$, 
	\[h(\calf):=\{p\in\beta G: p\subset\calf\}\]
	is a subsemigroup of $\beta G$ and $\calf\supset \fps$, then for any $G$-system $(X,G)$ and $x\in X$, the following assertions are equivalent:
	\begin{enumerate}
		\item $x$ is distal;
		\item $x$ is $\calf$-product recurrent, that is, for any $G$-system $(Y,G)$ and any $\calf$-recurrent point $y\in Y$, $(x,y)$ is recurrent in the product system $(X\times Y,G)$; 
		\item  for any $G$-system $(Y,G)$ and any $\calf$-recurrent point $y\in Y$, $(x,y)$ is $\calf$-recurrent in the product system $(X\times Y,G)$.
	\end{enumerate}
\end{thm}

The paper is organized as follows. 
To illustrate the core idea, 
in Section 2 we focus on topological dynamical systems and prove Theorem \ref{thm:main1}. 
The proof takes advantage of the order of natural numbers and is thus relatively straightforward. 
In the rest part of this paper, we consider general group actions and Ellis actions. 
In Section 3, we investigate some properties of several collections of subsets in a countably infinite discrete group $G$. 
In Section 4, for compact metric $G$-systems we provide combinatorial characterizations of the return time sets of $\calf$-recurrent points under the conditions (P1) and (P2) introduced in Section 3. 
We also present an application of product recurrence for $G$-systems. 
In Section 5, we recall some results about Stone-\v{C}ech compactification $\beta G$ of $G$ and prove the main result (Theorem~\ref{thm:main3}) of this paper, which can be regarded as a generalization of Theorem~\ref{thm:main1}.
In Section 6, we study $\beta G$-actions on compact Hausdorff spaces and prove Theorems~\ref{thm:main-result-S-product-recurrent} and \ref{thm:main-result-distal-product-rec}.

\section{Proof of Theorem \ref{thm:main1}}\label{section:thm1.2}
In this section, we focus on continuous maps acting on compact metric space and  devote to prove Theorem \ref{thm:main1}. 
It should be noted that the natural order of $\bbn_0$ plays a significant role in the proof of Theorem \ref{thm:main1}, whereas in the general case ($G$-system), the situation becomes more  complicated. 
To illustrate the core idea of the construction,  we decide to prove Theorem \ref{thm:main1} in a separate section, which may be of independent interest. 
We will try our best to make this section self-contained to ensure that readers can understand it independently. 
Readers are referred to Theorems \ref{thm:main2} and \ref{thm:main3} for the general case. 

In Subsection \ref{subsection:defn} we will discuss some equivalent definitions of quasi-central sets. 
For the proof of Theorem \ref{thm:main1}, readers may refer directly to Subsection \ref{subsection:proof}.

\subsection{Some equivalent definitions of quasi-central sets}\label{subsection:defn} 
First we introduce the structure of $\beta\bbn_0$.
Denote by $\mathcal{P}=\mathcal{P}(\bbn_0)$ the collection of all subsets of $\bbn_0$. A subset $\calf$ of $\mathcal{P}$ is called \emph{Furstenberg family 
(or just family)} if it is hereditary upward, i.e., $F_1\subset F_2$ and 
$F_1\in\calf$ imply $F_2\in\calf$. A family $\calf$ is called \emph{proper}
if it is neither empty nor all of $\mathcal{P}$. A family is called \emph{filter} when it is a proper family closed under intersection,i.e.,
if $F_1,F_2\in\calf$ then $F_1\cap F_2\in\calf$. A family is called \emph{ultrafilter} if it is a filter that are maximal with respect to inclusion.

Before going on, let us recall some notions.  
By a \emph{compact right topological semigroup}, we mean a triple $(E,\cdot,\mathcal{T})$, where $(E,\cdot)$ is a semigroup, and $(E,\mathcal{T})$ is a compact Hausdorff space, and for every $p\in E$, the right translation $\rho_p\colon S\to S$, $q\mapsto q\cdot p$ is continuous. 
If there is no ambiguous, we will say that  $E$, instead of the triple $(E,\cdot,\mathcal{T})$, is a compact right topological semigroup.
A nonempty subset $I$ of $E$ is called a \emph{left ideal} of $E$ if  $E \cdot I\subset I$, a \emph{right ideal} of $E$ if  $I \cdot E\subset I$ and an \emph{ideal} of $E$ if it is both a left ideal and a right ideal of $E$.
A \emph{minimal left ideal} is the left ideal that does not contain any proper left ideal. A \emph{minimal right ideal} is the right ideal that does not contain any proper right ideal.
An element $p\in E$ is called \emph{idempotent} if $p\cdot p=p$.
An idempotent $p\in E$ is called a \emph{minimal idempotent} if there exists a minimal left ideal $L$ of $E$ such that $p\in L$.  
Ellis-Namakura theorem reveals every compact right topological semigroup must contains an idempotent, see e.g.~\cite[Theorem 2.5]{HS12}.

Endowing $\bbn_0$ with the discrete topology, we take the points of the Stone-\v{C}ech compactification
$\beta\bbn_0$ of $\bbn_0$ to be the ultrafilter on $\bbn_0$.
For $A\subset\bbn_0$, let $\overline{A}=\{p\in\beta\bbn_0: A\in p\}$.
Then the sets $\{\overline{A}: A\subset\bbn_0 \}$ forms a
basis for the open sets (and a basis for the closed sets) of $\beta\bbn_0$.
Since $(\bbn_0,+)$ is a semigroup, we can extend the operation $+$ to $\beta\bbn_0$ as
\[p+q=\{F\subset \bbn_0: \{n\in\bbn_0: -n+F\in q\}\in p\}.\]
Then $(\beta\bbn_0, +)$ is a compact Hausdorff right topological semigroup
with $\bbn_0$ contained in the topological center of $\beta\bbn_0$.
That is, for each $p\in \beta\bbn_0$ the map $\rho_p: \beta\bbn_0\to\beta\bbn_0$, $q\mapsto q+p$ is continuous,
and for each $n\in\bbn_0$ the map $\lambda_n: \beta\bbn_0\to\beta\bbn_0$, $q\mapsto n+q$ is continuous.
It is well known that $\beta\bbn_0$ has a smallest ideal
$K(\beta\bbn_0)=\bigcup\{L: L $ is a minimal left ideal of $\beta\bbn_0\}
= \bigcup\{R: R $ is a minimal right ideal of $\beta\bbn_0\}$ (\cite[Theorem 2.8]{HS12}).
Let $p\in\beta \bbn_0$, $\{x_n\}_{n\in \bbn_0}$ be an indexed family in a compact Hausdorff space $X$ and $y\in X$. 
If for every neighborhood $U$ of $y$, $\{n\in \bbn_0\colon x_n\in U\}\in p$,
then we say that the \emph{$p$-limit} of $\{x_n\}_{n\in \bbn_0}$ is $y$, denoted by $p\text{-}\lim_{n\in \bbn_0}x_n=y$. 
As $X$ is a compact Hausdorff space, $p\text{-}\lim_{n\in \bbn_0}x_n$ exists and is unique. 

According to \cite[Definition 1.2]{HMS96}, we introduce the following original definition of quasi-central sets.
\begin{defn}
Let $F\subset\bbn_0$. Then $F$ is quasi-central if and only if there exists some idempotent 
$p\in \cl(K(\beta\bbn_0))$ with $F\in p$.
\end{defn}

We recall some classes of subsets of $\bbnz$. 
\begin{defn}
Let $A$ be a subset of $\bbnz$. 
\begin{enumerate}
	\item If for every $L\in \bbn$, 
	there exists $n\in \bbnz$ such that $\{n,n+1,\dotsc,n+L\}\subset A$, then we say that $A$ is \emph{thick}.
	\item If there exists $L\in \bbn$ such that for any $n\in \bbnz$, $\{n,n+1,\dotsc,n+L\}\cap A\neq\emptyset$,
	then we say that $A$ is \emph{syndetic}.
	\item If there exists a thick set $B\subset \bbnz$ and a syndetic $C\subset \bbnz$ such that $A=B\cap C$, then we say that $A$ is \emph{piecewise syndetic}.
\end{enumerate}
\end{defn}

Let $(X,\langle T_s \rangle_{s\in S})$ be a dynamical system defined in \cite{BH07} where $S$ is a semigroup. Note that when $S=\bbn_0$, the action is generated by a continuous evolution map 
$T$ and we simply write the dynamical system as $(X,T)$ in this section (the underlying space $X$ is a compact metric space). 
By the proof of \cite[Theorem 3.4]{BH07},  we have the following theorem, which is a dynamical characterization of quasi-central set.
\begin{thm}\label{thm:quasi=F}
Let $F\subset\bbn_0$. Then $F$ is quasi-central if and only if there exists a dynamical system $(X,T)$, points $x$ and $y$ of $X$, and a neighborhood $U$ of $y$ such that
\begin{enumerate}
\item for any neighborhood $V$ of $y$, 
$N((x,y),V\times V)$ is piecewise syndetic and
\item $N(x,U)=F$.
\end{enumerate}
\end{thm}

We will need the following equivalent characterizations of quasi-central sets.
\begin{prop}\label{prop:eq-quasi-central}
Let $F\subset\bbn_0$. Then the following assertions are equivalent:
\begin{enumerate}
    \item $F$ is quasi-central;
    \item there exists a dynamical system $(X,T)$, points $x$ and $y$ of $X$, and a neighborhood $U$ of $y$ such that
    \begin{enumerate}
	    \item for any neighborhood $V$ of $y$, $N((x,y),V\times V)$ is piecewise syndetic and
	    \item $N(x,U)\subset F\cup\{0\}$.
    \end{enumerate}
\item there exists a dynamical system $(X,T)$, points $x$ and $y$ of $X$, and a neighborhood $U$ of $y$ such that
    \begin{enumerate}
	\item for any neighborhood $V$ of $y$, $N((x,y),V\times V)$ is piecewise syndetic and
	\item $N((x,y),U\times U)\subset F\cup\{0\}$.
    \end{enumerate}
\end{enumerate} 
\end{prop}

\begin{proof}
(1)$\Rightarrow$(2).
It follows from Theorem \ref{thm:quasi=F}.

(2)$\Rightarrow$(3).
It follows from the fact $N((x,y),U\times U)\subset N(x,U)$.

(3)$\Rightarrow$(1).
By  \cite[Lemma 3.3]{BH07} one can pick an idempotent $p\in \cl(K(\beta\bbn_0))$ such that $p\text{-}\lim_{n\in\bbn_0}T^nx=p\text{-}\lim_{n\in\bbn_0}T^ny=y$.
For the neighborhood $U$ of $y$, $N((x,y),U\times U)=N(x,U)\cap
N(y,U)\in p$. 
Then $F\cup\{0\}\in p$. Since $K(\beta\bbn_0)\subset \beta\bbn_0\setminus\bbn_0$, $F\in p$.
By the definition $F$ is quasi-central.
\end{proof}

\subsection{Proof of Theorem \ref{thm:main1}}\label{subsection:proof}
In this subsection we will prove Theorem \ref{thm:main1}.
\begin{defn}
Let $(X,T)$ be a topological  dynamical system and $x\in X$. We say that $x$ is \emph{a piecewise syndetic recurrent point} if for any neighborhood $U$ of $x$, $N(x,U):=\{n\in\bbn_0:
T^nx\in U\}$ is a piecewise syndetic set.
\end{defn}

\begin{lem}\label{lem:a+1}
Let $(X,T)$ be a dynamical system, let $x,y\in X$, and assume that for every neighborhood $V$ of $y$, $N((x,y),V\times V)$ is piecewise syndetic in $\bbn_0$.  
Let $U$ be a neighborhood of $y$ and let $a\in\bbn$. There are a set $H$ which is thick in $\bbn_0$ and a set $S$ which is 
syndetic in $\bbn_0$ such that 
$H\cap S\subset N((x,y),U\times U)$ and $S\subset (a+1)\bbn$.
\end{lem}
\begin{proof}
By Proposition \ref{prop:eq-quasi-central},  $N((x,y), U\times U)$ is a quasi-central set. Then by \cite[Lemma 5.19.2]{HS12} or \cite[Proposition 6.7]{L12}, 
	$\frac{1}{a_1+1}N((x,y), \allowbreak U\times U)\cap \bbn$ is  piecewise syndetic  in $\bbn$, so is in $\bbn_0$.
	There exists a thick set $H'$ of $\bbn$ and a syndetic set 
	$S'$ of $\bbn$ such that \[
	    H'\cap S'=\frac{1}{a_1+1}N((x,y),U\times U)\cap \bbn.
	\]
   Let $H = \bigcup_{j=0}^{a_1} ((a_1+1)H' +j)$ and $S=(a_1+1) S'$. 
   Then $H$ is thick, $S$ is syndetic and
   \[
   H\cap S\subset (a_1+1)(H'\cap S')\subset N((x,y),U\times U).
   \] 
   This ends the proof of the lemma.
\end{proof}

Now we introduce the symbolic dynamical system $(\Sigma_2,\sigma)$.
Let 
$$\Sigma_2=\{0,1\}^{\mathbb{N}_0}=\{x_0x_1x_2\dotsc: x_i\in\{0,1\},i\in\bbn_0\},$$ 
endowed with the product topology on $\{0,1\}^{\mathbb{N}_0}$, while $\{0,1\}$ is endowed with the discrete topology.
A compatible metric $d$ on $\Sigma_2$ is defined by 
\[
d(x,y)= \begin{cases}
0,  &  \ x=y; \\
\frac{1}{2^k},  & k=\min\{i\in \mathbb{N}_0\colon x_i\neq y_i\}, \\
\end{cases}
\]
for any $x,y\in\Sigma_2$.
Then $(\Sigma_2,d)$ is a compact metric space. 
Define the \emph{shift map} as follows
\[\sigma\colon \Sigma_2 \to \Sigma_2,\ x_0x_1x_2\dotsc\mapsto  x_1x_2x_3\dotsc\]  
Then $(\Sigma_2,\sigma)$ is a topological dynamical system. 
Besides infinite symbolic sequences we consider also finite 
symbolic sequences or \emph{word} $u=u_0u_1\dotsc u_{n-1}$ where $u_i\in \{0,1\}$ for $i=0,\dotsc,n-1$.  
If $u=u_0u_1\dotsc u_{n-1}$ is a word of $\{0,1\}$, 
we define the 
\emph{cylinder} of $u$ as 
\[
[u]= \big\{v\in \Sigma_2 \colon v_i=u_i, \text{ for any } 0\leq i\leq n-1\big\}.
\]
Obviously $[u]$ is a clopen subset of $\Sigma_2$. 
Denote $\{0,1\}^n=\{x_0x_1\dotsc x_{n-1}:x_i\in \{0,1\}, 0\leq i\leq n-1\}$ and $\{0,1\}^*=\bigcup_{n=1}^\infty\{0,1\}^n$.
Then the collection of all cylinders
$\{[u]: u\in \{0,1\}^*\}$
forms a topological basis of the topology of $\Sigma_2$. In particular, for any $x=x_0x_1x_2\dotsc\in\Sigma_2$, 
we denote by $x|_{[i,j]}=x_i\dotsc x_j$ the word which occurs in $x$ between coordinates $i$ and $j$. Then we can consider the cylinder $[x|_{[i,j]}]$, i.e., $[x|_{[i,j]}]=\big\{v\in \Sigma_2 \colon v_s=x_s, \text{ for any } i\leq s\leq j\big\}$.
For any $x,y\in \Sigma_2$, 
$x|_{[i,j]}=y|_{[i,j]}$ means that the two words are identical, i.e.,
for any $s\in \{i,\dotsc, j\}$, $x_s=y_s$.

\begin{proof}[Proof of Theorem~\ref{thm:main1}]
(1) Since $x\in X$ is piecewise syndetic recurrent,
for
every neighborhood $V$ of $x$, $N(x,V)$ is a piecewise syndetic set. Then
for the system $(X,T)$,
$x\in X$ and a neighborhood $U$ of $x$, it satisfies that
\begin{enumerate}
	\item[(i)] for every neighborhood $V$ of $x$, $N((x,x),V\times V)= N(x,V)$ is piecewise syndetic;
	\item[(ii)] $N(x,U)= N(x,U)\cup \{0\}$.
\end{enumerate}  
Thus $N(x,U)$ is quasi-central.

(2) 
Let $F$ be a quasi-central subset of $\bbn_0$. 
By Proposition \ref{prop:eq-quasi-central}, there exists a topological dynamical system $(X,T)$, $x,y\in X$ and a neighborhood $U$ of $y$ such that 
	\begin{enumerate}
		\item[(i)] for every neighborhood $V$ of $y$, $N((x,y),V\times V)$ is piecewise syndetic in $\bbn_0$ and
		\item[(ii)] $N(x,U)\subset F\cup\{0\}$.
	\end{enumerate} 

    We shall show that for the symbolic dynamical system $(\Sigma_2,\sigma)$,
    there exists a point $z\in \Sigma_2$ which is a piecewise syndetic recurrent point such that $[1]$
    is a neighborhood of $z$ and $N(z,[1])\subset F\cup\{0\}$.
 
	Let $U_1=U$.  
	Since $N((x,y),U_1\times U_1)$ is piecewise syndetic in $\bbn_0$,
	pick a set $H_1$ which is thick in $\bbn_0$ and a set $S_1$ which is syndetic in $\bbn_0$
	such that $H_1\cap S_1= N((x,y),U_1\times U_1)$. 
	Pick a finite integer interval $I_1^{(1)}\subset H_1$ such that $I_1^{(1)}\cap S_1\neq\emptyset$, $\min I_1^{(1)}>1$
	and $|I_1^{(1)}|>1$, where $|\cdot|$ denote the cardinality of the set. Define $z^{(1)}\in \Sigma_2$ as follows:
	\begin{equation*}
	z^{(1)}(n)=\begin{cases}
	1, & n=0;\\
	1, & n\in I^{(1)}_1\cap S_1;\\
	0, & n\in \bbn_0\setminus \{\{0\}\cup (I^{(1)}_1\cap S_1)\}. 
	\end{cases}
	\end{equation*}
	Then $z^{(1)}(0)=1$, $z^{(1)}(1)=0$
	and $N(z^{(1)},[1])=\{0\}\cup  (I^{(1)}_1\cap S_1)$.
	Let $A_1= N(z^{(1)},[1])$ and let $a_1=\max A_1$. Then $A_1$ is a finite subset of $\bbn_0$ and $A_1\subset N((x,y),U_1\times U_1)\cup \{0\}$. 

Let $k\in\bbn$ and assume that we have chosen $\big\langle z^{(i)}\big\rangle_{i=1}^k$ in $\Sigma_2$, $\big\langle U_{i}\big\rangle_{i=1}^k$ neighborhood of 
$y$ in $X$,  $\big\langle A_{i}\big\rangle_{i=1}^k$, $\big\langle a_{i}\big\rangle_{i=1}^k$, $\big\langle H_{i}\big\rangle_{i=1}^k$, $\big\langle S_{i}\big\rangle_{i=1}^k$ and $\big\langle\big\langle I_{i}^{(j)}\big\rangle_{j=1}^i\big\rangle_{i=1}^k$ satisfying the following hypotheses for $i\in\{1,2,\dotsc,k\}$.

\begin{enumerate}
\item $A_i=N(z^{(i)},[1])\subset N((x,y),U_1\times U_1)\cup \{0\}$ and $a_i=\max A_i$;
\item if $i>1$, then $A_{i-1}\subset A_i$ and $a_{i-1}<a_i$;
\item if $i>1$, then $U_i=\bigcap_{j\in A_{i-1}} T^{-j} U_{1}$;
\item  $H_i$ is thick in $\bbn_0$, $S_i$ is syndetic in $\bbn_0$, and $H_i\cap S_i \subset N((x,y), U_i\times U_i)$; 
\item if $i>1$, then $S_i\subset(a_{i-1}+1)\bbn$;
\item if $1\leq j\leq i$, then $I_{i}^{(j)}$ is a finite interval,  $|I^{(j)}_i|>i$, $I^{(j)}_i\subset H_j$ and $ I_i^{(1)} \cap S_1\neq\emptyset$; 
\item if $i>1$, then $\min I_i^{(1)}>a_{i-1}$, 
and $\min I_i^{(2)}>\max  I_i^{(1)}$; 
\item  if $i>2$, then  $\min I_i^{(1)}>\max  I_{i-1}^{(i-1)}+a_{i-2}$,  $\min I_i^{(2)}>\max I_i^{(1)}$ and if $3\leq j\leq i$, then 
$\min I_i^{(j)}>\max  I_{i-1}^{(j-1)}+a_{j-2}$;
\item if $i>1$, then $z^{(i)}|_{[0,a_{i-1}]}=z^{(i-1)}|_{[0,a_{i-1}]}$; 
\item if $n\in I_i^{(1)}\cap S_1$, then $z^{(i)}(n)=1$;
\item if $2\leq j\leq i$ and $n\in I_i^{(j)}\cap S_{j}$, $z^{(i)}|_{[n,n+a_{j-1}]}=z^{(j-1)}|_{[0,a_{j-1}]}$; 		
\item if $i>1$ and $t\in \bbn_{0}\setminus ([0,a_{i-1}]\cup (I_i^{(1)}\cap S_1)\cup \bigcup_{j=2}^i \bigcup_{n\in I_i^{(j)}\cap S_{j}} [n,n+a_{j-1}]   )$, then $z^{(i)}(t)=0$.
\end{enumerate}
All hypotheses satisfied for $i=1$, all but $(1)$, $(4)$, $(6)$ and $(10)$
vacuously.

We now show that all hypotheses satisfied for $i=k+1$. 
Let $U_{k+1}=\bigcap_{j\in A_k}T^{-j}U_1$.
By hypothesis $(1)$, if $j\in A_k$, then 
$j\in  N((x,y),U_1\times U_1)\cup \{0\}$ so $T^jy\in U_1$. Therefore 
$U_{k+1}$ is an open neighborhood of $y$. 
By Lemma \ref{lem:a+1}, pick a thick subset $H_{k+1}$ of $\bbn_0$ and a syndetic subset $S_{k+1}$ of $\bbn_0$
such that $S_{k+1}\subset (a_k+1)\bbn$ and 
$H_{k+1}\cap S_{k+1}\subset N((x,y),U_{k+1}\times U_{k+1})$.

Take a finite interval $I_{k+1}^{(1)}$ in $H_{1}$ with $\min I_{k+1}^{(1)}>a_k$ such that $I_{k+1}^{(1)}\cap S_1\neq\emptyset$ and 
$\min I_{k+1}^{(1)}>\max I_{k}^{(k)}+u$ where 
\begin{equation}
\nonumber
u=
\begin{cases}
0,\ \ \ \text{if}\  k=1;\\
a_{k-1},\ \ \ \text{if}\ k>1.
\end{cases}   
\end{equation}
For $j\in\{2,3,\dotsc, k+1\}$ pick a finite interval $I_{k+1}^{(j)}$ in $H_j$ such that 
$|I_{k+1}^{(j)}|>k+1$, $\min I_{k+1}^{(j)}>\max I_{k+1}^{(j-1)}$ and if $j\geq 3$, then 
$\min I_{k+1}^{(j)}>\max I_{k+1}^{(j-1)}+a_{j-2}$. 

We claim that we can define $z^{(k+1)}\in\Sigma_2$ as required by 
hypotheses $(9)-(12)$ for $i=k+1$. That is,
\begin{enumerate}
\item [(9)] $z^{(k+1)}|_{[0,a_{k}]}=z^{(k)}|_{[0,a_{k}]}$; 
\item [(10)] if $n\in I_{k+1}^{(1)}\cap S_1$, then $z^{(k+1)}(n)=1$;
\item [(11)] if $2\leq j\leq k+1$ and $n\in I_{k+1}^{(j)}\cap S_{j}$, $z^{(k+1)}|_{[n,n+a_{j-1}]}=z^{(j-1)}|_{[0,a_{j-1}]}$; 	
\item [(12)] if $t\in \bbn_{0}\setminus ([0,a_{k}]\cup (I_{k+1}^{(1)}\cap S_1)\cup \bigcup_{j=2}^{k+1} \bigcup_{n\in I_{k+1}^{(j)}\cap S_{j}} [n,n+a_{j-1}]   )$, then $z^{(k+1)}(t)=0$.
\end{enumerate}

By the construction of $I_{k+1}^{(j)}$, $j=1,\dots,k+1$, 
we have 
$\min I_{k+1}^{(j)}>\min I_{k+1}^{(1)}>a_{k}$ for $j\in\{1,2,\dotsc,k+1\}$, which implies that 
$(9)$ cannot conflict with $(10)$ or $(11)$. 

To see that $(10)$ cannot conflict with any part of $(11)$, let $j\in \{1,2,\dotsc,k+1\}$, let $m\in I_{k+1}^{(j)}\cap S_j$ and let $t\in [0,a_{j-1}]$. Then 
$m+t\geq \min I_{k+1}^{(j)}>\max I_{k+1}^{(1)}\geq n$.

Finally, we show that any part of $(11)$ cannot conflict with each other.
Suppose we have 
$2\leq j\leq l\leq k+1$, 
$n\in I_{k+1}^{(j)}\cap S_j$, $m\in I_{k+1}^{(l)}\cap S_l$, $t\in [0,a_{j-1}]$ and $s\in [0,a_{l-1}]$
such that $n+t=m+s$. Assume first that $j=l$. If $n=m$,
then $t=s$ and there is no conflict. So suppose without loss of generality that $n<m$. Then 
$n,m\in (a_{j-1}+1)\bbn$ so pick $b<c$ in $\bbn$ such that $n=(a_{j-1}+1)b$ and $m=(a_{j-1}+1)c$.
Then $n+t=(a_{j-1}+1)b+t<(a_{j-1}+1)c\leq m+s$, a contradiction. Thus we must have $j<l$ so $l\geq 3$.
Then $m+s\geq \min I_{k+1}^{(l)}>\max I_{k+1}^{(l-1)}+a_{l-2}\geq \max I_{k+1}^{(j)}+a_{j-1}\geq n+t$, a contradiction.

Let $A_{k+1}=N(z^{k+1},[1])$ and let $a_{k+1}=\max A_{k+1}$.
All hypotheses are satisfied directly except $(1)$ and $(2)$.
To see that $A_k\subset A_{k+1}$, let $n\in A_k$.
Then $n\leq a_k$ so by hypothesis $(9)$, $z^{(k+1)}(n)=z^{(k)}(n)=1$. Also by $(6)$, $I_{k+1}^{(1)}\cap S_1\neq\emptyset$ so by
$(10)$ $a_{k+1}\geq \min (I_{k+1}^{(1)}\cap S_1)$ and 
$\min (I_{k+1}^{(1)}\cap S_1)\geq \min I_{k+1}^{(1)}>a_k$ by $(7)$. Thus hypothesis $(2)$ holds.

To verify hypothesis $(1)$ we need to show that 
$N(z^{(k+1)},[1])\subset N((x,y),U_1\times U_1)\cup \{0\}$. 
So let $m\in N(z^{(k+1)},[1])$. If $m\in [0,a_k]$, then 
$m\in A_k\subset N((x,y),U_1\times U_1)\cup \{0\}$.
If $m\in I_{k+1}^{(1)}\cap S_1$, then $m\in H_1\cap S_1\subset N((x,y),U_1\times U_1)$. So assume that we have $2\leq j\leq k+1$, $n\in I_{k+1}^{(j)}\cap S_j$, and $t\in A_{j-1}$ such that $m=n+t$.
By hypothesis $(6)$ and $(4)$, $n\in N((x,y),U_j\times U_j)$ so $T^n x\in U_j$ and $T^n y\in U_j$. By hypothesis $(3)$, 
$T^t(T^n x)\in U_1$ and $T^t(T^n y)\in U_1$
so $m=n+t\in N((x,y),U_1\times U_1)$. The inductive
construction is complete.

We now establish some facts.
\begin{itemize}
\item[(a)] if $1\leq r<j\leq i$, then for each $n\in I_{j}^{(r+1)}\cap S_{r+1}$, \[z^{(i)}|_{[n,n+a_r]}=z^{(j)}|_{[n,n+a_r]}=z^{(r)}|_{[0,a_r]}.\]
\end{itemize}
 
To establish $(a)$, let $1\leq r<j\leq i$, let $n\in I_{j}^{(r+1)}\cap S_{r+1}$ and let $t\in [0,a_r]$. By hypothesis $(11)$, $z^{(j)}(n+t)=z^{(r)}(t)$. Now 
$z^{(j)}(n+a_r)=z^{(r)}(a_r)=1$ so $n+a_r\in A_j$ and thus
$n+a_r\leq a_j$. Then by hypotheses $(2)$ and $(9)$, 
$z^{(i)}(n+t)=z^{(j)}(n+t)=z^{(r)}(t)$.

\begin{itemize}
\item [(b)] if $1\leq r<j\leq i$, then $I_j^{(r+1)}\cap S_{r+1}\subset N(z^{(i)}, [z^{(r)}|_{[0,a_r]}])$.
\end{itemize}

To establish $(b)$, let $1\leq r<j\leq i$ and let $n\in I_{j}^{(r+1)}\cap S_{r+1}$. Then by $(a)$, for each $t\in [0,a_r]$, $\sigma^n(z^{(i)})(t)=z^{(i)}(n+t)=z^{(r)}(t)$
so $n\in N(z^{(i)}, [z^{(r)}|_{[0,a_r]}])$ as required.
\smallskip

Since $\left\langle z^{(i)}\right\rangle_{i=1}^\infty$ is a sequence in compact space $\Sigma_2$, we may pick a cluster point 
$z\in \Sigma_2$ of the sequence $\left\langle z^{(i)}\right\rangle_{i=1}^\infty$.

\begin{itemize}
\item [(c)] For each $j\in\bbn$, $z|_{[0,a_j]}=z^{(j)}|_{[0,a_j]}$.
\end{itemize}

To establish $(c)$, let $j\in\bbn$ and let $t\in [0,a_j]$. Since $z$ is a cluster point of the  sequence $\left\langle z^{(i)}\right\rangle_{i=1}^\infty$ and $[z|_{[0,a_j]}]$ is a neighborhood of $z$, we can pick 
$i>j$ such that $z^{(i)}\in [z|_{[0,a_j]}]$. Then 
$z^{(i)}|_{[0,a_j]}=z|_{[0,a_j]}$ and by hypotheses $(2)$ and $(9)$, $z^{(j)}|_{[0,a_j]}=z^{(i)}|_{[0,a_j]}=z|_{[0,a_j]}$.
\smallskip

As a consequence of $(c)$, for each $r\in\bbn$, $[z^{(r)}|_{[0,a_r]}]$ is a neighborhood of $z$. So $\{[z^{(r)}|_{[0,a_r]}]:r\in\bbn\}$ is a neighborhood basis for $z$.
\begin{enumerate}
\item [(d)] If $1\leq r< i$, then $N(z^{(i)}, [z^{(r)}|_{[0,a_r]}])\subset N(z, [z^{(r)}|_{[0,a_r]}])$.
\end{enumerate}

To establish $(d)$, let $1\leq r< i$ and $n\in N(z^{(i)}, [z^{(r)}|_{[0,a_r]}])$, then for any $t\in [0,a_r]$,
\[\sigma^n(z^{(i)})(t)=z^{(i)}(n+t)=z^{(r)}(t).\]
In particular $z^{(i)}(n+a_r)=z^{(r)}(a_r)=1$ so 
$n+a_r\in A_i$ and thus $n+a_r\leq a_i$. By $(c)$, 
$z|_{[0,a_i]}=z^{(i)}|_{[0,a_i]}$ so $\sigma^n(z)(t)=z(n+t)=
z^{(i)}(n+t)=z^{(r)}(t)$. Thus $n\in N(z, [z^{(r)}|_{[0,a_r]}])$ as claimed.
\smallskip

Now we claim that $z$ is a piecewise syndetic recurrent point of $\Sigma_2$. To see this, let $R$
be a neighborhood of $z$ and pick $r\in\bbn$ such that $[z^{(r)}|_{[0,a_r]}]\subset R$. As $S_{r+1}$ is syndetic and $\bigcup_{i=r+1}^\infty I_i^{(r+1)}$ is thick, $S_{r+1}\cap (\bigcup_{i=r+1}^\infty I_i^{(r+1)})$ is piecewise syndetic and 
\begin{align*}
S_{r+1}\cap \biggl(\bigcup_{i=r+1}^\infty I_i^{(r+1)}\biggr)&= \bigcup_{i=r+1}^\infty (S_{r+1}\cap I_i^{(r+1)}) \\
& \subset \bigcup_{i=r+1}^\infty N(z^{(i)},[z^{(r)}|_{[0,a_r]}])\subset N(z, [z^{(r)}|_{[0,a_r]}]),
\end{align*} 
where the first inclusion holds by $(b)$ and the second inclusion holds by $(d)$. So $z$ is a piecewise syndetic recurrent point of $\Sigma_2$.

By $(c)$ $[1]$ is a neighborhood of $z$. We conclude the proof by showing that $N(z,[1])\subset F\cup\{0\}$.
If $n\in N(z,[1])$ and $a_i>n$ then by $(c)$ $z(n)=z^{(i)}(n)$ so $N(z,[1])\subset \bigcup_{i=1}^\infty N(z^{(i)},[1])$. By hypothesis $(1)$, for each $i\in\bbn$, $N(z^{(i)},[1])\subset N((x,y),U_1\times U_1)\cup \{0\}\subset N(x,U)\cup\{0\}$ so $N(z,[1])\subset F\cup\{0\}$.
\end{proof}

\section{Subsets in a countable infinite group} 

In this section we investigate some classes of subsets in a countable infinite discrete group. 
We propose two abstract properties (P1) and (P2) for a Furstenberg family which we will use in Section 3 to characterize recurrent time sets. 
We will verify that the collection of all piecewise syndetic sets and the collection of all infinite sets satisfy the two abstract properties. 
If the group is amenable, the collection of all sets with positive upper density (with positive upper Banach density, respectively) also satisfies the two abstract properties.

Let $G$ be a countable infinite discrete group with identity $e$.
Denote by $\calpg$ and $\calpfg$ the collections of all subsets of $G$ and all nonempty finite subsets of $G$ respectively.
Let $\calf\subset \calpg\setminus\{\emptyset\}$.
If for any $F\in \calf$, $F\subset H\subset G$ implies $H\in\calf$,
then we say that $\calf$ is a \emph{Furstenberg family (or just family)}. 
A Furstenberg family $\calf$ is said to be \emph{proper} if 
it is a proper subset of $\calpg\setminus\{\emptyset\}$.
For a Furstenberg family $\calf$, 
the \emph{dual family} of $\calf$, denote by $\calf^*$, 
is
\[\{F\in\mathcal{P}(G):F\cap F'\not=\emptyset,\ \text{for any} \ F'\in\calf\}.\]
Note that $\calf^*=\{F\in\mathcal{P}(G):
G\setminus F\not\in \calf\}$.
A Furstenberg  family $\calf$ is called a \emph{filter} if  $A,B\in \calf$ imply $A\cap B\in\calf$.
A \emph{ultrafilter} is a filter which is not properly contained in any other filter.
A Furstenberg family $\calf$ has \emph{Ramsey property} if whenever $A\in\calf$ and  $A=A_1\cup A_2$ there exists some  $i\in\{1,2\}$ such that $A_i\in\calf$. 
It is easy to see that a Furstenberg family $\calf$ has the Ramsey property if and only if the dual family $\calf^*$ is a filter.

Let $A$ be a subset of $G$. 
\begin{enumerate}
\item If for every $K\in \calpfg$, 
there exists $g\in G$ such that $Kg\subset A$, then we say that $A$ is \emph{thick}.
\item If there exists $K\in \calpfg$ such that for any $g\in G$, $Kg\cap A\neq\emptyset$ (i.e. $G=K^{-1} A$),
then we say that $A$ is \emph{syndetic}.
\item If there exists a thick set $B\subset G$ and a syndetic $C\subset G$ such that $A=B\cap C$, then we say that $A$ is \emph{piecewise syndetic}.
\end{enumerate}
Denote by $\ft$, $\fs$, $\fps$, $\finf$ the collection of all thick, syndetic, piecewise syndetic and infinite subsets of $G$. 

We say that a Furstenberg family $\calf$ satisfies (P1) if 
for any $A\in\calf$ there exists a sequence 
$\{A_n\}_{n=1}^\infty$ in $\calpfg$ such that 
 \begin{enumerate}
     \item for every $n\in\bbn$, $A_n\subset A$;
     \item for every $n,m\in\bbn$ with $n\neq m$, 
     $A_n\cap A_m=\emptyset$;
     \item for every strictly increasing sequence $\{n_k\}_{k=1}^\infty$
     in $\bbn$, $\bigcup_{k=1}^\infty A_{n_k}\in\calf$,
 \end{enumerate}
and (P2) if 
for any $F\in\calf$ and any $K\in\calpfg$,
there exists a subset $F'$ of $F$ such that $F'\in\calf$ and for any 
distinct $f_1,f_2\in F'\cup\{e\}$, $Kf_1\cap Kf_2=\emptyset$.

First we need the following lemma.
\begin{lem}\label{lem:property}
For every $F,H\in\calpfg$ if $|H|>|F|^2$
there exists $h\in H$ such that $F\cap Fh=\emptyset$.
\end{lem}
\begin{proof}
For any $f_1,f_2\in F$, let
$B(f_1,f_2)=\{h\in H\colon f_1=f_2 h\}$.
As $G$ is a group, each $B(f_1,f_2)$ is the empty set or a singleton.
If for every $h\in H$, $F\cap Fh\neq\emptyset$,
then $\bigcup_{f_1,f_2\in F}B(f_1,f_2)=H$.
As $|H|>|F|^2$, there exist $f_1,f_2\in F$,
such that $B(f_1,f_2)$ contains at least two points.
This is a contraction.
\end{proof}

The following result must be folklore. We provide a proof for the sake of completeness.
\begin{lem}\label{lem:t-P1-s-P2}
$\ft$ satisfies (P1).
\end{lem}
\begin{proof}
We need the following claim.

\medskip 

\noindent\textbf{Claim}
Let $F$ be a thick set.
Fix $K\in\calpfg$, then $\{g\in G\colon Kg\subset F\}$ is a thick set.
\begin{proof}
For any $H\in\calpfg$, $KH\in \calpfg$.
As $F$ is thick, there exists $h\in G$ such that $KHh\subset F$.
Then $Hh\subset \{g\in G\colon Kg\subset F\}$. So $\{g\in G\colon Kg\subset F\}$ is a thick set. 
\end{proof}

Now fix a thick set $A$.
As $G$ is countable, there exists a sequence $\{G_n\}_{n=1}^\infty$ in $\calpfg$ such that $G_n\subset G_{n+1}$ and $\bigcup_{n=1}^\infty G_n =G$.
As $A$ is thick, there exists $g_1\in G$ such that $G_1 g_1\subset A$.
Let $A_1= G_1g_1$.
Let $B_2=A_1\cup G_2$.
By the claim, $\{g\in G\colon B_2g\subset A\}$ is thick.
By Lemma \ref{lem:property}, 
there exists $g_2\in \{g\in G\colon B_2g\subset A\}$ such that $B_2\cap B_2g_2=\emptyset$.
Let $A_2=G_2g_2$. 

By induction, we construct two sequences $\{A_n\}$, $\{B_n\}$ in $\calpfg$ and a sequence $\{g_n\}$ in $G$ such that for any $n\geq 2$,
\begin{enumerate}
    \item $B_n=\bigcup_{i=1}^{n-1}A_i \cup G_n$;
    \item $B_n g_n\subset A$;
    \item $B_n \cap B_ng_n=\emptyset$;
    \item $A_n=G_n g_n$.
\end{enumerate}
    Then for any $n\in \bbn$, $A_n\subset B_n g_n\subset A$; for any $n,m\in \bbn$ with $n\neq m$, without loss of generality assume that $n>m$, $A_n\cap A_m\subset A_n\cap B_n\subset B_n g_n\cap B_n=\emptyset$; Since $G_n\subset G_{n+1}$ and $\bigcup_{n=1}^\infty G_n =G$, for every strictly increasing sequence $\{n_k\}_{k=1}^\infty$
     in $\bbn$, $\bigcup_{k=1}^\infty A_{n_k}\in\calf_t$. Thus $\ft$ satisfies (P1). 
\end{proof}

In \cite{XY22} Xu and Ye showed that $\fs$ satisfies (P2).
Here we have the following sufficient condition for a Furstenberg family to satisfy (P2).

\begin{prop}\label{prop:P2}
Let $\calf$ be a proper Furstenberg family in $\mathcal{P}(G)\setminus\{\emptyset\}$. 
If $\calf$ has the Ramsey property and for every $A\in\calf$ and $g\in G$, $gA\in \calf$, then $\calf$ satisfies (P2).
\end{prop}
\begin{proof}
We first show the following Claim.

\medskip 
\noindent\textbf{Claim:} For every $A\in\calf$ and $K\in\calpfg$, $A\setminus K\in \calf$.
\begin{proof}[Proof of the Claim]
Let $A\in\calf$ and $K\in\calpfg$.
As $\calf$ has the Ramsey property and $A=(A\cap K)\cup (A\setminus K)$, either $A\cap K\in\calf$ or $A\setminus K\in\calf$. Now we assume that $A\cap K\in\calf$ and
write the finite $A\cap K$ as $\{k_1,k_2,\dotsc,k_n\}$.
By the Ramsey property of $\calf$ again, there exists some $1\leq i\leq n$ such that $\{k_i\}\in\calf$.
For every $g\in G$, $g\{k_i\}=\{gk_i\}\in\calf$. 
As $\calf$ is a Furstenberg family, $\calf=\calpg\setminus\{\emptyset\}$, which contradicts that $\calf$ is proper.
Therefore, $A\setminus K\in\calf$.
\end{proof}

Now Fix $A\in\calf$ and $K\in\calpfg$.
Let 
	\[
	\calb=\{B\subset A\colon \text{for any 
		distinct $b_1,b_2\in B\cup\{e\}$, $Kb_1\cap Kb_2=\emptyset$}
	\}.
	\]
	By the Claim, $A$ is infinite.
	By Lemma \ref{lem:property}, there exists
	$h\in A\setminus \{e\}$ such that $K\cap Kh=\emptyset$, then $\{h\}\in\calb$, which implies that 
	$\calb$ is not empty.
	By the Zorn's Lemma, pick $B\in\calb$ which is maximal with respect to the inclusion relation. If $D\in \calb$ then also $D\cup \{e\}\in\calb$ and since $B\in\calb$ is maximal with respect to the inclusion relation, $e\in B$.

Now we will show that $B\in\calf$. 
	For any $a\in A$, there exists $b\in B$ such that 
	$Ka\cap Kb\neq\emptyset$. (For otherwise there is $a\in A$ such that for any $b\in B$, we have $Ka\cap Kb=\emptyset$. so $a\notin B$, $B\subsetneq B\cup \{a\}\in\mathcal{B}$, contradicting the maximality of set $B$). 
	Then $a\in K^{-1}Kb$. 
	This shows that $A\subset K^{-1}K B$.
Then $K^{-1}K B\in\calf$ as $A\in\calf$.
As $\calf$ has the Ramsey property and $K^{-1}K$ is finite, there exists some $g\in K^{-1}K$ such that 
$gB\in\calf$. Then $B=g^{-1}(gB)\in\calf$. 
\end{proof}

It is easy to see that $\finf$ satisfies the properties (P1) and (P2).
Now we show that $\fps$ also satisfies the properties (P1) and (P2).
\begin{lem} \label{lem:fps-P1-P2}
$\fps$ satisfies  (P1) and (P2).
\end{lem}
\begin{proof}
(1) $\fps$ satisfies (P1).

Let $F\in\fps$.
By the definition of $\fps$, 
there exists a thick set $A\subset G$ and a syndetic set $B\subset G$ such that $F=A\cap B$.
By Lemma \ref{lem:t-P1-s-P2} $\calf_t$ satisfy (P1), then there exists a sequence 
$\{A_n\}_{n=1}^\infty$ in $\calpfg$ such that 
\begin{itemize}
	\item for every $n\in\bbn$, $A_n\subset A$;
	\item for every $n,m\in\bbn$ with $n\neq m$, 
	$A_n\cap A_m=\emptyset$;
	\item for every strictly increasing sequence $\{n_k\}_{k=1}^\infty$
	in $\bbn$, $\bigcup_{k=1}^\infty A_{n_k}\in\calf_t$.
\end{itemize}
Let $F_n=A_n\cap B$ for $n\in\bbn$. 
Then $\{F_n\}_{n=1}^\infty$ is the sequence that satisfies (P1) for $F$. By the
arbitrariness of $F$, $\fps$ satisfies (P1).
 
(2) $\fps$ satisfies (P2).

Let $F\in\fps$.
By the definition of $\fps$, 
there exists a thick set $A\subset G$ and a syndetic set $B\subset G$ such that $F=A\cap B$.
For any $K\in\calpfg$,
by \cite[Lemma 2.7]{XY22} 
$\calf_{s}$ satisfy (P2), then 
there exists a subset $B'$ of $B$ such that $B'\in\calf_s$ and for any 
distinct $b_1,b_2\in B'\cup\{e\}$, $Kb_1\cap Kb_2=\emptyset$.
Let $F'=A\cap B'$, then $F'\subset F$ and $F'\in\fps$.
For any 
distinct $f_1,f_2\in F'\cup\{e\}$, $f_1,f_2\in B'\cup\{e\}$, thus 
$Kf_1\cap Kf_2=\emptyset$.
By the arbitrariness of $F$, $\fps$ satisfies (P2).
\end{proof}

A F{\o}lner sequence of a group $G$ can be used to define the density of a set $A\subset G$ in a way analogous to the definition given for a subset of non-negative integers of natural density. 

For any nonempty subsets $A,B$ in $G$. Denote $A\Delta B=(A\setminus B)\cup (B\setminus A)$. It is easy to verify that for any nonempty subsets $A,B,C,D$ in $G$, $(A\setminus B)\Delta (C\setminus D)\subset (A\Delta C)\cup (B\Delta D)$. 

\begin{defn}
Let $G$ be a countable infinite discrete group and $\{F_n\}$ be a sequence of nonempty finite subsets of $G$.
We say that $\{F_n\}$ is a \emph{F{\o}lner sequence} if for any $g\in G$, we have 
\[\lim_{n\to\infty}\frac{|(gF_n)\Delta F_n|}{|F_n|}=0,
\]
It is obviously that if $\{F_n\}$ is a F{\o}lner sequence, then $\lim_{n\to\infty}|F_n|=+\infty$. 

A countable infinite discrete group $G$ is called an \emph{amenable group} if 
there exists some F{\o}lner sequence $\{F_n\}$ in $G$. 
\end{defn}

\begin{defn}
Let $G$ be a countable infinite discrete amenable group and $\{F_n\}$ be a  F{\o}lner sequence in $G$.
For a subset $A$ of $G$, the \emph{upper density} of $A$ with respect to the F{\o}lner sequence $\{F_n\}$ 
is defined by
\[
    \bar{d}_{\{F_n\}}(A)=\limsup_{n\to\infty} \frac{1}{|F_n|}|F_n\cap A|.
\]
It is obvious that $0\leq \bar{d}_{\{F_n\}}(A) \leq 1$. For a given F{\o}lner sequence $\{F_n\}$, 
denote 
\[\fpud^{\{F_n\}}=\{A\subset G\colon \bar{d}_{\{F_n\}}(A)>0\}.\]

The \emph{upper Banach density} of $A$ is defined by
\[
    d^*(A)=\sup\{\bar{d}_{\{F_n\}}(A)\colon  
    \{F_n\} \text{ is a F{\o}lner sequence in }G\}.
\]
It is obvious that $0\leq d^*(A) \leq 1$. 
Denote $\fpubd=\{A\subset G\colon d^*(A)>0\}$.
\end{defn}

In the following we show that if $G$ is an amenable group and $\{F_n\}$ is an F{\o}lner sequence in $G$, then
$\fpud^{\{F_n\}}$ and $\fpubd$ satisfy the properties (P1) and (P2).

\begin{lem}\label{lem:fpud-fpubd-P1-P2}
Let $G$ be an amenable group and  $\{F_n\}$ be a  F{\o}lner sequence in $G$.
Then $\fpud^{\{F_n\}}$ and $\fpubd$ satisfy (P1) and (P2).
\end{lem}
\begin{proof}
(1) $\fpud^{\{F_n\}}$ satisfies (P1).
	
	Let $A\in\fpud^{\{F_n\}}$, then 
	\[
	\bar{d}_{\{F_n\}}(A)=\limsup_{n\to\infty} \frac{1}{|F_n|}|F_n\cap A|>0.
	\]
	Then there exists
	a F{\o}lner subsequence
	$\{F_n'\}\subset \{F_n\}$ such that
	\[
	\lim_{n\to\infty} \frac{1}{|F_n'|}|F_n'\cap A|>0.
	\]
	Without loss of generality we assume that
	$|F_n'|>(n+1)(|F_1'|+\dotsb+|F_{n-1}'|)$ for any $n\geq 2$.
	Define $E_1:=F_1'$ and 
	$E_n:=F_n'\setminus (F_1'\cup\dotsb\cup F_{n-1}')$ for any $n\geq 2$.
	It is clear that $E_i\cap E_j=\emptyset$ for any distinct $i,j\in \bbn$.

 \medskip 
 
 \noindent\textbf{Claim:} $\{E_n\}$ is a F{\o}lner sequence and  $\bar{d}_{\{E_n\}}(A)= \bar{d}_{\{F_n'\}}(A)$.

\begin{proof}[Proof of the Claim]
 Since
\begin{align*}
		(gE_n)\Delta E_n
		&=\big((gF_n')\setminus g(F_1'\cup F_2'\cup \dotsb \cup F_{n-1}')\big)\Delta \big(F_n'\setminus (F_1'\cup F_2'\cup \dotsb \cup F_{n-1}')\big)\\
		&\subset\big((gF_n')\Delta F_n'\big)\cup \big(g(F_1'\cup F_2'\cup \dotsb \cup F_{n-1}')\Delta (F_1'\cup F_2'\dotsc \cup F_{n-1}')\big), 
\end{align*}
we have 
\begin{align*}
\lim_{n\to\infty}\frac{|(gE_n)\Delta E_n|}{|E_n|}
		&\leq\lim_{n\to\infty}\frac{|(gF_n')\Delta F_n'|}{|E_n|}+
		\lim_{n\to\infty}\frac{|g(F_1'\cup \dotsc \cup F_{n-1}')\Delta (F_1'\cup \dotsc \cup F_{n-1}')|}{|E_n|}\\
		&\leq \lim_{n\to\infty}\frac{|(gF_n')\Delta F_n'|}{|F_n'|}+
		\lim_{n\to\infty}\frac{2|F_1'\cup \dotsc \cup F_{n-1}'|}{n(|F_1'|+\dotsc+|F_{n-1}'|)}=0.
\end{align*}
So by the definition $\{E_n\}$ is a F{\o}lner sequence. 

It is easy to verify that
\[\bar{d}_{\{E_n\}}(A)= 
\limsup_{n\to\infty} \frac{|(F_n'\setminus (F_1'\cup\dotsb\cup F_{n-1}'))\cap A|}{|F_n'\setminus (F_1'\cup\dotsb\cup F_{n-1}')|}
= \limsup_{n\to\infty} \frac{|F_n'\cap A|}{|F_n'|}=\bar{d}_{\{F_n'\}}(A). \qedhere \] 

Similarly, we can verify that for every strictly increasing sequence $\{n_k\}_{k=1}^\infty$ in $\bbn$, $\{E_{n_k}\}$ is a F{\o}lner sequence and  $\bar{d}_{\{E_{n_k}\}}(A)= \bar{d}_{\{F_{n_k}'\}}(A)$. 
\end{proof}

Let  $A_n:=E_n\cap A$.
 Then $A_n\subset A$ and
$A_n\cap A_{m}=\emptyset$ for every $n,m\in\bbn$ with $n\neq m$.
For any strictly increasing sequence $\{n_k\}$ in $\bbn$, 
\begin{align*}
	\bar{d}_{\{F_n\}}\biggl(\bigcup_{k=1}^\infty A_{n_k}\biggr)
	\geq\bar{d}_{\{F_n'\}}\biggl(\bigcup_{k=1}^\infty A_{n_k}\biggr)
	&=\limsup_{n\to\infty} \frac{1}{|F_n'|}\biggl|F_n'\cap \bigcup_{k=1}^\infty A_{n_k}\biggr|\\
		&\geq \limsup_{k\to\infty} \frac{1}{|E_{n_k}|}|E_{n_k}\cap A|\\
	&= \bar{d}_{\{E_{n_k}\}}(A). 
 \end{align*}	
 By the claim, $\bar{d}_{\{E_{n_k}\}}(A)= \bar{d}_{\{F_{n_k}'\}}(A)=\bar{d}_{\{F_n'\}}(A)>0$. So $\bigcup_{k=1}^\infty A_{n_k}\in \calf_{pud}^{\{F_n\}}$. 
Thus $\{A_n\}$ is the sequence satisfies (P1) for $A$. 
By the arbitrariness of $A$, $\fpud^{\{F_n\}}$ satisfies (P1). 

(2) $\fpubd$ satisfies (P1).
Let $A\in\fpubd$.
There exists a F{\o}lner sequence $\{F_n\}$ such that $\bar{d}_{\{F_n\}}(A)>0$. Then it follows from the proof of $\fpud^{\{F_n\}}$ satisfies (P1).

(3) It is easy to verify that $\fpud^{\{F_n\}}$ and $\fpubd$ satisfy all the conditions in Proposition \ref{prop:P2}. 
Then $\fpud^{\{F_n\}}$ and $\fpubd$ satisfy (P2).
\end{proof}

\section{Return time sets and product recurrence for \texorpdfstring{$G$}{G}-systems on compact metric spaces}
In this section we study recurrent time sets of points with some special recurrent property in a $G$-system $(X,G)$. 
Note that in this section, we always assume that $X$ is a compact metric space. 
Using the abstract properties (P1) and (P2) of Furstenberg families in Section 3 we give combinatorial characterizations of return time sets of $\calf$-recurrent points. We also apply those results to the study of product recurrence.

First we introduce $G$-system and recall some definitions.
By a compact (metric) \emph{$G$-system}, we mean a triple $(X,G,\Pi)$, where $X$ is a compact (metric) space with a metric $d$, $G$ is a countable infinite discrete group with an identity $e$ and 
$\Pi:G\times X\to X$ is a continuous map satisfying
$\Pi(e,x)=x$, for all $x\in X$
and $\Pi(h,\Pi(g,x))=\Pi(hg,x)$, for all $x\in X$, $h,g\in G$.
For convenience, we will use the pair $(X,G)$ instead of $(X,G,\Pi)$ to denote the $G$-system, and $gx:=\Pi(g,x)$ if the map $\Pi$ is unambiguous.
For two systems $(X,G)$ and $(Y,G)$, 
there is a natural product system $(X\times Y,G)$
as $g(x,y)=(gx,gy)$ for every $g\in G$ and $(x,y)\in X\times Y$.
A nonempty closed $G$-invariant subset $Y\subseteq X$ defines naturally a subsystem $(Y, G)$ of $(X, G)$.
A $G$-system $(X,G)$ is called \emph{minimal} if it contains no proper subsystem.
Each point belonging to some minimal subsystem of $(X,G)$ is called a \emph{minimal point}.
By the Zorn's Lemma, every $G$-system has a minimal subsystem.

Let $(X,G)$ be a $G$-system. For a point $x\in X$ and open subsets $U, V \subset X$, define
\[
    N(x,U)=\{g\in G\colon gx\in U\},
\]
and 
\[ N(U,V)=\{g\in G\colon gU\cap V\not=\emptyset\}.\]
The \emph{orbit} of a point $x\in X$ is the set
$Gx=\{gx:\ g\in G\}$, and the \emph{orbit closure} is $\overline{Gx}$.
Any point with dense orbit is called \emph{transitive}.
It is easy to see that  $(X,G)$ is minimal if and only if every point in $X$ is transitive. 
A $G$-system $(X,G)$ is called \emph{transitive} if for any nonempty open sets $U$ and $V$ of $X$, $N(U,V)\not=\emptyset$.
A point $x\in X$ is called \emph{recurrent} if for any neighborhood $U$ of $x$, $N(x,U)$ is infinite, 
and \emph{almost periodic} (it is also known as uniformly recurrent) if for any neighborhood $U$ of $x$, $N(x,U)$ is a syndetic set. 
It is well known that a point $x$ is almost periodic if and only if the system $(\overline{Gx},G)$ is minimal.

\begin{defn}
Let $G$ be a countable infinite discrete group.
For a sequence $\{p_i\}_{i=1}^\infty$ in $G$, we define the \emph{finite product} of $\{p_i\}_{i=1}^\infty$  by 
\[
FP(\{p_i\}_{i=1}^\infty) =\biggl\{\prod_{i\in \alpha} p_i\colon  \alpha
\text{ is a nonempty finite subset of }  \bbn \biggr\},
\]
where $\prod_{i\in \alpha} p_i$ is the product in increasing order of indices.
A subset $F$ of $G$
is called an \emph{IP-set} if
there exists a sequence $\{p_i\}_{i=1}^\infty$ in $G$ such that $FP(\{p_i\}_{i=1}^{\infty})$ is infinite and
$FP(\{p_i\}_{i=1}^{\infty})\subset F$.
Denote by $\fip$ the collection of all IP-subsets of $G$.
\end{defn}

Let $(X,G)$ be a $G$-system, $x\in X$ and $\calf\subset \calpg$ be a Furstenberg family.
We say that $x$ is \emph{$\calf$-recurrent} 
if for every neighborhood $U$ of $x$, $N(x,U)\in\calf$. We also called $\calf_{ps}$-recurrent point is \emph{piecewise syndetic recurrent} point. 
We will further study recurrent time sets of $\calf$-recurrent points.
First we introduce the Bernoulli shift $(\Sigma_2,  G)$ and symmetrically $\calf$-sets which are closely related to the corresponding recurrent time sets.

For a countable infinite discrete group $G$ with identity $e$, let $\Sigma_2=\{0,1\}^G$, endowed with the product topology on $\{0,1\}^{G}$, while $\{0,1\}$ is endowed with the discrete topology.
An element of $\Sigma_2$ is a function $z: G\to \{0,1\}$. 
Enumerate $G$ as $\{g_i\}_{i=0}^\infty$ with $g_0=e$.
A compatible metric $d$ on $\Sigma_2$ is defined by 
\begin{equation*}
d(z_1,z_2)=\begin{cases}
0, & z_1= z_2;\\
\frac{1}{2^k}, & k=\min \{i\in \bbn_0\colon z_1(g_i)\neq z_2(g_i)\},
\end{cases}
\end{equation*}
for any $z_1,z_2\in\Sigma_2$.
Then $(\Sigma_2, d)$ is a compact metric space. 

For any $K\in \calpfg$ and $u\in \{0,1\}^K$, define a \emph{cylinder} as follows: 
\[
 [u]=\{z\in\Sigma_2:z(g)=u(g)\ \text{for}\ g\in K\}.
\]
Then the collection of all cylinders
$\{[u]: u\in \{0,1\}^K\ \text{for some}\ K\in\calpfg\}$
forms a topological basis of the topology of $\Sigma_2$. For every $z\in \Sigma_2$ and $K\in \calpfg$, denote $z|_K\in \{0,1\}^K$ with $z|_K(g)=z(g)$ for every $g\in K$, then we can consider the 
cylinder $[z|_K]$. 
For convenience, we denote
$[1]=\{z\in \Sigma_2\colon z(e)=1\}$. 

For $g\in G$, define $T_g: \Sigma_2\to \Sigma_2$ by:
\[T_gz(t)=z(tg),\ \text{for any}\ t\in G.\]
Then $(\Sigma_2, (T_g)_{g\in G})$ is a $G$-system, which is called the \emph{symbolic dynamical system over $G$}. We briefly denote $(\Sigma_2, (T_g)_{g\in G})$ as $(\Sigma_2, G)$. 

For a subset $F\subset G$, let $\mathbf{1}_F\in\Sigma_2$ be the characteristic function of $F$,  that is, 
\begin{equation*}
\mathbf{1}_{F}(g)=\begin{cases}
1, & g\in F;\\
0, & \text{otherwise}. 
\end{cases}
\end{equation*}

In \cite{KRS22} Kennedy et al.\@ introduced the concept of symmetrically syndetic set and showed that 
the dual family of symmetrically syndetic sets is the family of dense orbit sets, which answered Question 9.6 in \cite{GTWZ21}. 
Recall that a subset $A\subset G$ is \emph{symmetrically syndetic} if for every pair of nonempty finite subsets $F_1\subset A$ and $F_2\subset G\setminus A$, 
the set 
\[
    \bigcap_{f_1\in F_1}f_1^{-1}A  \cap \bigcap_{f_2\in F_2}f_2^{-1}(G\setminus A)
\] 
is syndetic.
In \cite{XY22} Xu and Ye showed a subset of $G$ is symmetrically syndetic if and only if it is a return time set of an almost periodic point in the Bernoulli shift $(\Sigma_2,  G)$.

Similar to the symmetrically syndetic set, a general symmetrically set can be defined.
Given a Furstenberg family $\calf$ over $G$, 
a subset $A \subset G$ is  a \emph{symmetrically $\calf$-set}, if for any nonempty finite subsets $F_1\subset A$ and $F_2 \subset  G\setminus A$, 
\[
\bigcap_{f_1\in F_1}f_1^{-1}A  \cap \bigcap_{f_2\in F_2}f_2^{-1}(G\setminus A)\in\calf. 
\]

We show that the family of sets containing a symmetrically $\calf$-set coincides the collection of the return time sets of $\calf$-recurrent points.

\begin{prop}\label{prop:rec-sigma}
Let $G$ be a countable infinite discrete group with identity $e$ and $\calf\subset \calpg$ be a Furstenberg family. For a given subset $F$ of $G$ with $e\in F$, the following assertions are equivalent:
\begin{enumerate}
    \item $F$ contains a symmetrically $\calf$-set $F'$ with $e\in F'$. 
    \item there exists an $\calf$-recurrent point $x\in \{0,1\}^G$ with $x\in [1]$
    such that $N(x,[1])\subset F$;
    \item there exists a $G$-system $(X,G)$, an $\calf$-recurrent point $x\in X$ and a neighborhood $U$ of $x$ such that $N(x,U)\subset F$;
\end{enumerate}
\end{prop}
\begin{proof}
(1)$\Rightarrow$(2).
As $G$ is countable, there exists a sequence $\{G_n\}_{n=1}^\infty$ in $\calpfg$ such that
$e\in G_1$,
$G_n\subset G_{n+1}$ and $\bigcup_{n=1}^\infty G_n =G$.
Consider the Bernoulli shift $(\Sigma_2,G)$.
Define
\begin{equation*}
\mathbf{1}_{F'}(g)=\begin{cases}
1, & g\in F';\\
0, & \text{otherwise}. 
\end{cases}
\end{equation*} 
For any $n\in \mathbb{N}$, let 
\[
I_n=F'\cap G_n, \ 
J_n=G_n\setminus F'. 
\]
Then for any $n\in \mathbb{N}$, $I_n\sqcup J_n=G_n$, $[\mathbf{1}_{F'}|_{G_n}]=[\mathbf{1}_{F'}|_{I_n}]\cap [\mathbf{1}_{F'}|_{J_n}]$, 
\[
N(\mathbf{1}_{F'}, [\mathbf{1}_{F'}|_{G_n}])= \bigcap_{f_1\in I_n}f_1^{-1}F'\cap \bigcap_{f_2\in J_n}f_2^{-1}(G\setminus F')\in\calf. 
\]
Obviously that $\{[\mathbf{1}_{F'}|_{G_n}]\colon n\in \bbn_0\}$ is a neighborhood basis of $\mathbf{1}_{F'}$. By the arbitrariness of $n$, 
this shows that $\mathbf{1}_{F'}$ is an $\calf$-recurrent point in $(\Sigma_2,G)$.
It is clear that $N(\mathbf{1}_{F'},[1])=F'\subset F$.

(2)$\Rightarrow$(3). It is clear. 

(3)$\Rightarrow$(1). 
As $G$ is countable, there exists a sequence $\{G_n\}_{n=1}^\infty$ in $\calpfg$ such that
$e\in G_1$,
$G_n\subset G_{n+1}$ and $\bigcup_{n=1}^\infty G_n =G$.
According to $(3)$, 
there exists a $G$-system $(X,G)$, an $\calf$-recurrent point $x$ and a neighborhood $U$ of $x$ such that $F\supset N(x,U)$.
Since $G$ is countable,
$Gx$ is 
countable, we can  choose a neighborhood $V$ of $x$ such that 
$\overline{V}\subset U$ and for any $g\in G$,
either $gx\in V$ or 
$gx\in X\setminus \overline{V}$.

Let $F':=N(x,V)$. Then $e\in F'\subset N(x,U)$. 
Now it is sufficient to show that 
$F'$ is a symmetrically $\calf$-set.
For any $g\in G$, we can
choose a neighborhood 
$W_g$ of $x$ with 
$W_g\subset V$
such that
if $gx\in V$ then $gW_g\subset V$ and
if $gx\in X\setminus \overline{V}$ then $gW_g\subset X\setminus \overline{V}$. 
For any finite set $G_n$, 
$\bigcap_{g\in G_n}W_g$
is a neighborhood of $x$.
Denote $W:=\bigcap_{g\in G_n}W_g$. Then $N(x,W)\subset F$ and $N(x,W)\in\calf$.
Let $I_n=G_n\cap F'$, $J_n=G_n\setminus F'$.
We have
\[N(x,W)\subset \bigcap_{f_1\in I_n}f_1^{-1}F'  \cap \bigcap_{f_2\in J_n}f_2^{-1}(G\setminus F')\in \calf. \]
Thus 
$F'$ is a symmetrically $\calf$-set. 
\end{proof}

By the proof of Proposition~\ref{prop:rec-sigma}, we have the following consequence.

\begin{coro}
Let $G$ be a countable infinite discrete group with identity $e$ and $\calf\subset \calpg$ be a Furstenberg family. 
For a given subset $F$ of $G$ with $e\in F$,  the following assertions are equivalent:
\begin{enumerate}
    \item $F$ is a symmetrically $\calf$-set; 
    \item there exists an $\calf$-recurrent point $x\in \{0,1\}^G$ such that $N(x,[1])= F$; 
\end{enumerate}
\end{coro}

Though Proposition~\ref{prop:rec-sigma} connects the recurrent time sets of $\calf$-recurrent points with symmetrically $\calf$-sets, usually it is not easy to verify whether a set is a symmetrically $\calf$-set.
Under the conditions (P1) and (P2) introduced in Section 3, 
we have the following combinatorial characterization of recurrent time sets of $\calf$-recurrent points, which is the main result in this section.

\begin{thm}\label{thm:main2}
Let $G$ be a countable infinite discrete group with identity $e$ and $\calf\subset \calpg$ be a Furstenberg family satisfying (P1) and (P2). For a given $F\in\calf$ with $e\in F$, 
the following assertions are equivalent:
\begin{enumerate}
    \item there exists a $G$-system $(X,G)$, an $\calf$-recurrent point $x\in X$ and a neighborhood $U$ of $x$ such that $N(x,U)\subset F$;
    \item there exists a decreasing sequence $\{F_n\}$ of subsets of $F$ in $\calf$ such that for any $n\in\bbn$ and $f\in F_n$ there exists 
    $m\in\bbn$ such that $f F_m \subset F_n$.
\end{enumerate}
\end{thm}

\begin{proof}
(1)$\Rightarrow$(2).
According to (1), there exists a $G$-system $(X,G)$, an $\calf$-recurrent point $x\in X$ and a neighborhood $U$ of $x$ such that $N(x,U)\subset F$. 
Then there exists $\delta>0$, such that $B(x,\delta)\subset U$.

For $n\in\bbn$, define $F_n:=N(x,B(x,\frac{\delta}{n}))$. It is clear that $F_{n+1}\subset F_n \subset F$ and $F_n\in\calf$ for $n\in\bbn$. 
Now fix $F_n$ and $f\in F_n$, then
$fx\in B(x,\frac{\delta}{n})$ and
$x\in f^{-1}B(x,\frac{\delta}{n})$.
It is clear that $f^{-1}B(x,\frac{\delta}{n})$ is a neighborhood of $x$, thus
there exists
$m\in\bbn$ such that 
$B(x,\frac{\delta}{m})\subset f^{-1}B(x,\frac{\delta}{n})$.
Then we have $fN(x,B(x,\frac{\delta}{m}))\subset N(x,B(x,\frac{\delta}{n}))$, i.e.
$fF_m\subset F_n$.

(2)$\Rightarrow$(1).
As $G$ is countable, fix a sequence $\{G_n\}_{n=1}^\infty$ in $\calpfg$
such that $G_1=\{e\}$, 
$G_n\subset G_{n+1}$ and $\bigcup_{n=1}^\infty G_n =G$.
Without loss of generality assume that $e\in F_n$ for any $n\in\bbn$. 
Let $m_1=1$, $F'_{1}=F_1$ and $B_1=\{e\}$.
Since $\calf$ satisfies the condition (P1), for $F_1\in \calf$, 
	there exists a sequence 
	$\{C_n^{(1)}\}_{n=1}^\infty$ in $\calpfg$ such that 
		\begin{itemize}
		\item for every $n\in\bbn$, 
  $C_n^{(1)}\subset F_1$;
		\item for every $n,n'\in\bbn$ with $n\neq n'$, 
		$C_n^{(1)}\cap C_{n'}^{(1)}=\emptyset$;
		\item for every strictly increasing sequence $\{n_i\}_{i=1}^\infty$
		in $\bbn$, $\bigcup_{i=1}^\infty C_{n_i}^{(1)}\in\calf$.
	\end{itemize}
 Let $A^{(1)}_1=C^{(1)}_1$. Consider the symbolic dynamical system $(\Sigma_2,G)$.
First, we define $z^{(1)}\in\Sigma_2$ as follows:
\begin{equation*}
z^{(1)}(g)=\begin{cases}
1, & g=e;\\
1, & g\in A_1^{(1)};\\
0, & \text{otherwise}. 
\end{cases}
\end{equation*}
Let $k\in\bbn$ and assume that we have chosen $\{z^{(i)}\}_{i=1}^k$ in $\Sigma_2$, 
$\{F_{m_i}\}_{i=1}^k$ and 
$\{F_{m_i}'\}_{i=1}^k$ in $\calf$, $\{B_n\}_{n=1}^k$, $\{C_n^{(i)}\}_{n=1}^\infty$, $i=1,\dotsc,k$ and $\{A_n^{(i)}\}_{n=1}^k$, $i=1,\dotsc,k$ in $\mathcal{P}_f(G)$, $\{\{t(j,i)\}_{j=1}^i\}_{i=2}^k$ in $\bbn$ 
 satisfying the following hypotheses for $i\in\{1,2,\dotsc,k\}$.
\begin{enumerate}
	\item if $i>1$, then $B_i= N(z^{(i-1)},[1])\cup G_i$;
	\item $N(z^{(i)},[1])\in\mathcal{P}_f(F_1)$;
	\item if $i>1$, then $N(z^{(i)},[1])= N(z^{(i-1)},[1])\cup A_{i}^{(1)}\cup  \bigcup_{j=1}^{i-1}(N(z^{(j)},[1])A_{i}^{(j+1)})$;
	\item $F_{m_i}'\subset F_{m_i}$;
	\item if $i>1$, then $N(z^{(i-1)},[1]) F_{m_i}\subset F_1$;
	\item for any distinct $f_1,f_2\in F_{m_i}'$, $B_if_1\cap B_if_2=\emptyset$.
    \item for every $n\in\bbn$, $C_n^{(i)}\subset F_{m_i}'$;
	\item for every $n,n'\in\bbn$ with $n\neq n'$, $C_n^{(i)}\cap C_{n'}^{(i)}=\emptyset$; 
	\item for every strictly increasing sequence $\{n_t\}_{t=1}^\infty$ in $\bbn$, $\bigcup_{t=1}^\infty C_{n_t}^{(i)}\in\calf$;
		\item 
		$t(1,1)=1$; 
		\item if $i\geq 2$ and $1\leq j\leq i-1$, then 
		$t(j,i)>t(j,i-1)$; 
		\item if $i\geq 2$, then $t(i,i)>i-1$;
		
		\item if $i\geq 2$ and $1\leq j\leq i-1$, then 
		$A^{(i)}_j=C^{(i)}_j$; 
		 
		\item if $1\leq j\leq i$, then 
		$A^{(j)}_i=C^{(j)}_{t(j,i)}$;
		
		\item if $i\geq 2$, $C_{t(1,i)}^{(1)}\cap B_i=\emptyset$,
		
		$C_{t(2,i)}^{(2)}\cap (B_2^{-1}B_i\cup B_2^{-1}A_i^{(1)})=\emptyset$,
		
		$\dotsc$, 
		
		$C_{t(i,i)}^{(i)}\cap (B_i^{-1}B_i\cup B_i^{-1}A_i^{(1)}\cup B_i^{-1}B_2A_i^{(2)} \dotsb\cup  
		B_i^{-1}B_{i-1}A_i^{(i-1)})=\emptyset$;
		
		\item if $i\geq 3$, $C_{t(1,i)}^{(1)}\cap (\bigcup_{t=2}^{i-1}(B_t A^{(t)}_t\cup\dotsb\cup B_t A^{(t)}_{i-1}))=\emptyset$,
		
		$C_{t(2,i)}^{(2)}\cap (\bigcup_{t=2}^{i-1}(B_2^{-1}B_t A^{(t)}_t\cup\dotsb\cup B_2^{-1}B_t A^{(t)}_{i-1}))=\emptyset$,
		
		$\dotsc$,
		
		$C_{t(i,i)}^{(i)}\cap (\bigcup_{t=2}^{i-1}(B_i^{-1}B_t A^{(t)}_t\cup\dotsb\cup B_i^{-1}B_t A^{(t)}_{i-1}))=\emptyset$;  

	\item if $i>1$, then $z^{(i)}|_{B_i}=z^{(i-1)}|_{B_i}$;
	
	\item if $g\in A_i^{(1)}$, then $z^{(i)}(g)=1$;

	\item if $2\leq j\leq i$, $h\in B_j$ and $g\in hA_i^{(j)}$, then
	$z^{(i)}(g)=z^{(j-1)}(h)$;
	
	\item if $i>1$ and $g\in G\setminus (B_i\cup A_i^{(1)}\cup \bigcup_{j=2}^i B_j A_i^{(j)})$, then  $z^{(i)}(g)=0$.
\end{enumerate}
All hypotheses are satisfied for $i=1$, all but (2), (4), (6), (7),
(8), (9), (14) and (18) vacuously.

We now show that all hypotheses satisfied for $i=k+1$. 
By hypotheses (2), 
$N(z^{(k)},[1])\in \mathcal{P}_f(F_1)$.
 For any $f\in N(z^{(k)},[1])$, by (2) there exists $m=m(f)\in\bbn$ such that $f F_m\subset F_1$. Let $m_{k+1}=\max \{m(f)\colon f\in N(z^{(k)},[1])\}$. 
Since $\{F_n\}$ is a decreasing sequence, $fF_{m_{k+1}}\subset F_1$ for every $f\in N(z^{(k)},[1])$.

Let $B_{k+1}=N(z^{(k)},[1])\cup G_{k+1}$.
By the condition (P2), for $F_{m_{k+1}}\in \calf$ and $B_{k+1}\in \mathcal{P}_f(G)$, there exists $F_{m_{k+1}}'\subset F_{m_{k+1}}$ with $F_{m_{k+1}}'\in\calf$ such that for any distinct $f_1,f_2\in F_{m_{k+1}}'$, $B_{k+1}f_1\cap B_{k+1}f_2=\emptyset$.
Since $F_{m_{k+1}}'\in\calf$,  again by the condition (P1), 
there exists a sequence $\{C_n^{(k+1)}\}_{n=1}^\infty$ in $\calpfg$ such that
\begin{itemize}
	\item for every $n\in\bbn$, $C_n^{(k+1)}\subset F_{m_{k+1}}'$;
	\item for every $n,n'\in\bbn$ with $n\neq n'$, 
	$C_n^{(k+1)}\cap C_{n'}^{(k+1)}=\emptyset$;
	\item for every strictly increasing sequence $\{n_t\}_{t=1}^\infty$
	in $\bbn$, $\bigcup_{t=1}^\infty C_{n_t}^{(k+1)}\in\calf$.
\end{itemize}
Let $A^{(k+1)}_j=C^{(k+1)}_j$ for $1\leq j\leq k$.
Since $B_{k+1}\in \mathcal{P}_f(G)$ and
$B_j A^{(j)}_j\cup\dotsb\cup B_j A^{(j)}_{k}\in \mathcal{P}_f(G)$ for
$k\geq 2$, $j=2,\dotsc,k$ and the elements in $\{C_n^{(1)}\}_{n=1}^\infty$ are pairwise disjoint, there exists $t(1,k+1)>t(1,k)$ such that $C_{t(1,k+1)}^{(1)}\cap B_{k+1}=\emptyset$ and $C_{t(1,k+1)}^{(1)}\cap (\bigcup_{j=2}^{k}(B_j A^{(j)}_j\cup\dotsb\cup B_j A^{(j)}_{k}))=\emptyset$ for $k\geq 2$. Let $A_{k+1}^{(1)}=C^{(1)}_{t{(1,k+1)}}$.
Similarly there exists
$t(j,k+1)>t(j,k)$ for $2\leq j\leq k$ such that
\[C^{(j)}_{t{(j,k+1)}}\cap (B_{j}^{-1}B_{k+1}\cup B_{j}^{-1}B_1A_{k+1}^{(1)}\cup B_{j}^{-1}B_2A_{k+1}^{(2)} \dotsb\cup  
B_{j}^{-1}B_{j-1}A_{k+1}^{(j-1)})=\emptyset\]
and 
\[C^{(j)}_{t{(j,k+1)}}\cap (\bigcup_{t=2}^{k}(B_{j}^{-1}B_t A^{(t)}_t\cup\dotsb\cup B_{j}^{-1}B_t A^{(t)}_{k}))=\emptyset.\]  
Let $A^{(j)}_{k+1}=C^{(j)}_{t(j,k+1)}$ for $2\leq j\leq k$.
And there exists $t{(k+1,k+1)}>k$ such that
\[C^{(k+1)}_{t{(k+1,k+1)}}\cap (B_{k+1}^{-1}B_{k+1}\cup B_{k+1}^{-1}B_1A_{k+1}^{(1)}\cup B_{k+1}^{-1}B_2A_{k+1}^{(2)} \dotsb\cup  
B_{k+1}^{-1}B_{k}A_{k+1}^{(k)})=\emptyset\]
and 
\[C^{(k+1)}_{t{(k+1,k+1)}}\cap (\bigcup_{t=2}^{k}(B_{k+1}^{-1}B_t A^{(t)}_t\cup\dotsb\cup B_{k+1}^{-1}B_t A^{(t)}_{k}))=\emptyset\ \text{for} \ k\geq 2.\]
Let $A_{k+1}^{(k+1)}=C^{(k+1)}_{t{(k+1,k+1)}}$.

We claim that we can define $z^{(k+1)}\in\Sigma_2$ as required by 
hypotheses $(17)-(20)$ for $i=k+1$. That is, 
\begin{enumerate}
\item[(17)] $z^{(k+1)}|_{B_{k+1}}=z^{(k)}|_{B_{k+1}}$;

\item[(18)] if $g\in A_{k+1}^{(1)}$, then $z^{(k+1)}(g)=1$;

\item[(19)] if $2\leq j\leq k+1$, $h\in B_{j}$ and $g\in hA_{k+1}^{(j)}$, then
$z^{(k+1)}(g)=z^{(j-1)}(h)$;

\item[(20)] if $g\in G\setminus \{B_{k+1}\cup A_{k+1}^{(1)}\cup \bigcup_{j=2}^{k+1} B_j A_{k+1}^{(j)}\}$, then  $z^{(k+1)}(g)=0$.
\end{enumerate}

By the construction of $B_{k+1}$, $C^{(1)}_{t{(1,k+1)}}$ and $A_{k+1}^{(1)}$, 
we have 
$C^{(1)}_{t{(1,k+1)}}\cap B_{k+1}=\emptyset$ and $A_{k+1}^{(1)}=C^{(1)}_{t{(1,k+1)}}$, 
thus $A_{k+1}^{(1)}\cap B_{k+1}=\emptyset$
which implies that 
(17) cannot conflict with (18). 

For $1\leq j\leq k+1$, by the construction of $B_{k+1}$, $C^{(j)}_{t{(j,k+1)}}$ and $A_{k+1}^{(j)}$, $C^{(j)}_{t{(j,k+1)}}\cap B_{j}^{-1}B_{k+1}=\emptyset$ and 
$A_{k+1}^{(j)}=C^{(j)}_{t{(j,k+1)}}$, thus $B_{k+1}\cap B_jA_{k+1}^{(j)}=\emptyset$ for $2\leq j\leq k+1$, which implies that (17) cannot conflict with (19).

For $1\leq j\leq k+1$, by the construction of $B_{k+1}$, $C^{(j)}_{t{(j,k+1)}}$ and $A_{k+1}^{(j)}$, 
$C_{t(j,k+1)}^{(j)}\cap B_j^{-1}A_{k+1}^{(1)}=\emptyset$ for $2\leq j\leq k+1$ and $A_{k+1}^{(j)}=C_{t(j,k+1)}^{(j)}$, thus 
$A_{k+1}^{(1)}\cap B_jA_{k+1}^{(j)}=\emptyset$ for $2\leq j\leq k+1$, which implies that $(18)$ cannot conflict with any part of (19).

Finally, we show that any part of (19) cannot conflict with each other.
By the construction of $B_{k+1}$, $C^{(j)}_{t{(j,k+1)}}$ and $A_{k+1}^{(j)}$, 
$C^{(j)}_{t{(j,k+1)}}\cap  (B_{j}^{-1}B_1A_{k+1}^{(1)}\cup B_{j}^{-1}B_2A_{k+1}^{(2)} \dotsb\cup  
B_{j}^{-1}B_{j-1}A_{k+1}^{(j-1)})=\emptyset$
for $2\leq j\leq k+1$. Therefore 
for any $2\leq j\neq j'\leq k+1$, 
$B_jA_{k+1}^{(j)}\cap B_{j'}A_{k+1}^{(j')}=\emptyset$.

Now all hypotheses are satisfied directly for $i=k+1$ except (2) and (3).
By the construction of $z^{(k+1)}$, 
\[N(z^{(k+1)},[1])=N(z^{(k)},[1])\cup A_{k+1}^{(1)}\cup\bigcup_{j=1}^{k}N(z^{(j)},[1])A_{k+1}^{(j+1)},\]
which implies that the hypothesis $(3)$ holds for $i=k+1$.

By the hypothesis (2) for $i=k$, $N(z^{(k)},[1])\in\mathcal{P}_f(F_1)$.
By the hypotheses $(4)$, $(7)$ and $(14)$,  $A_{k+1}^{(1)}=C_{t(1,k+1)}^{(1)}\subset F'_{m_1}\subset F_{m_1}=F_1$. Since $\{C_n^{(1)}\}_{n=1}^\infty$ is in $\mathcal{P}_f(G)$,  $A_{k+1}^{(1)}\in \mathcal{P}_f(F_1)$.
By the hypothesis (5), 
$N(z^{(j)},[1]) F_{m_{j+1}}\subset F_1$ for $j=1,\dotsc,k$. 
By the hypotheses (4), $(7)$ and $(14)$,  $A_{k+1}^{(j+1)}=C_{t(j+1,k+1)}^{(j+1)}\subset F'_{m_{j+1}}\subset F_{m_{j+1}}$ for $j=1,\dotsc,k$. 
Thus for $j=1,\dotsc,k$, $N(z^{(j)},[1])A_{k+1}^{(j+1)}\subset F_1$.
By the hypothesis (2) for $i=1,\dotsc,k$ and since
$A_{k+1}^{(j+1)},j=1,\dotsc,k$ is in $\mathcal{P}_f(G)$,
we have $N(z^{(j)},[1])A_{k+1}^{(j+1)}\in \mathcal{P}_f(F_1)$ for $j=1,\dotsc,k$.
In conclusion, 
\[N(z^{(k+1)},[1])=N(z^{(k)},[1])\cup A_{k+1}^{(1)}\cup\bigcup_{j=1}^{k}N(z^{(j)},[1])A_{k+1}^{(j+1)}\in \mathcal{P}_f(F_1),\]
which implies that the hypothesis (2) holds for $i=k+1$.

We now establish some facts.
\begin{enumerate}
\item[(\romannumeral1)] if $1\leq r<j$, then for each $h\in B_{r+1}$ and each $g\in A_{r+1}^{(r+1)}\cup A_{r+2}^{(r+1)}\cup \dotsc\cup A_{j}^{(r+1)}$, $z^{(j)}(hg)=z^{(r)}(h)$.
\end{enumerate}

By the hypothesis (19), for each $h\in B_{r+1}$ and each $g\in A_{j}^{(r+1)}$,
$z^{(j)}(hg)=z^{(r)}(h)$.
If $j=r+1$, then the proof is finished. Otherwise 
$j>r+1\geq 2$ and thus $j\geq 3$, to see that 
for each $h\in B_{r+1}$ and each $g\in A_{j-1}^{(r+1)}$,
$z^{(j)}(hg)=z^{(r)}(h)$.
We will first show that for each $h\in B_{r+1}$ and each $g\in A_{j-1}^{(r+1)}$, 
$z^{(j)}(hg)=z^{(j-1)}(hg)$.
By the hypothesis (17), $z^{(j)}|_{B_j}=z^{(j-1)}|_{B_j}$. 
By the hypothesis (1),
$B_j= N(z^{(j-1)},[1])\cup G_j$.
So 
$z^{(j-1)}{(hg)}=1$ implies $z^{(j)}(hg)=1$ for $g\in A_{j-1}^{(r+1)}$ and $h\in B_{r+1}$.
It is sufficient to show that
$z^{(j-1)}{(hg)}=0$ implies $z^{(j)}(hg)=0$ for $g\in A_{j-1}^{(r+1)}$ and $h\in B_{r+1}$. 
To prove this we note that
by the hypotheses (14) and (16), 
$A_j^{(1)}\cap B_{r+1}A_{j-1}^{(r+1)}=\emptyset$,
$B_tA_j^{(t)}\cap B_{r+1}A_{j-1}^{(r+1)}=\emptyset$ for $2\leq t\leq j$. Now by the hypothesis (19), for each $h\in B_{r+1}$ and each $g\in A_{j-1}^{(r+1)}$,
$z^{(j)}(hg)=z^{(j-1)}(hg)=z^{(r)}(h)$.
If $j-1=r+1$ then the proof is finished.
Otherwise $j-1>r+1\geq 2$ and thus $j\geq 4$, again we can show that 
for each $h\in B_{r+1}$ and each $g\in A_{j-2}^{(r+1)}$,
$z^{(j)}(hg)=z^{(r)}(h)$.
By induction the proof is finished.

\smallskip

Since $\{z^{(i)}\}_{i=1}^\infty$ is a sequence in compact space $\Sigma_2$, we may pick a cluster point 
$z\in \Sigma_2$ of the sequence $\{ z^{(i)}\}_{i=1}^\infty$.

\begin{enumerate}
	\item[(\romannumeral2)] For each $j\in\bbn$, $z|_{B_{j+1}}=z^{(j)}|_{B_{j+1}}$.
\end{enumerate}

To establish (\romannumeral2), let $j\in\bbn$ and let $g\in B_{j+1}$. Since $z$ is a cluster point of the  sequence $\{z^{(i)}\}_{i=1}^\infty$ and $[z|_{B_{j+1}}]$ is a neighborhood of $z$, we can pick 
$i>j$ such that $z^{(i)}\in [z|_{B_{j+1}}]$. Then 
$z^{(i)}|_{B_{j+1}}=z|_{B_{j+1}}$. By the construction $B_n\subset B_{n+1}$ for any $n\in\bbn$ and $\bigcup_{n=1}^\infty B_n\supset \bigcup_{n=1}^\infty G_n=G$.
So by hypotheses $(17)$, $z^{(j)}|_{B_{j+1}}=z^{(i)}|_{B_{j+1}}=z|_{B_{j+1}}$.
\smallskip

As a consequence of (\romannumeral2), for each $r\in\bbn$, $[z^{(r)}|_{B_{r+1}}]$ is a neighborhood of $z$ so $\{[z^{(r)}|_{B_{r+1}}]:r\in\bbn\}$ is a neighborhood basis for $z$.

\begin{enumerate}
	\item[(\romannumeral3)] If $1\leq r< i$, then $A_{r+1}^{(r+1)}\cup A_{r+2}^{(r+1)}\cup \dotsc\cup A_{i}^{(r+1)}\subset N(z, [z^{(r)}|_{B_{r+1}}])$.
\end{enumerate}

To establish (\romannumeral3), for any $g\in A_{r+1}^{(r+1)}\cup A_{r+2}^{(r+1)}\cup \dotsc\cup A_{i}^{(r+1)}$ and for any $h\in B_{r+1}$,
if $z^{(i)}(hg)=1$, then $hg\in N(z^{(i)},[1])\subset B_{i+1}$.
By (\romannumeral1), $A_{r+1}^{(r+1)}\cup A_{r+2}^{(r+1)}\cup \dotsc\cup A_{i}^{(r+1)}\subset N(z^{(i)}, [z^{(r)}|_{B_{r+1}}])$, then $z^{(i)}(hg)=z^{(r)}(h)$.
By (\romannumeral2), $z|_{B_{i+1}}=z^{(i)}|_{B_{i+1}}$, thus 
we have $z(hg)=z^{(i)}(hg)=z^{(r)}(h)$.

If $z^{(i)}(hg)=0$ and $hg\in B_{i+1}$, then we still have $z(hg)=z^{(i)}(hg)=z^{(r)}(h)$. 
If $z^{(i)}(hg)=0$ and $hg\not\in B_{i+1}$, since $B_n\subset B_{n+1}$ for any $n\in\bbn$ and $\bigcup_{n=1}^\infty B_n\supset \bigcup_{n=1}^\infty G_n=G$, $hg\in B_{t}$ for some $t>i+1$. Note that $1\leq r<i<i+1<t$, by (\romannumeral1), 
\begin{align*}
    N(z^{(t)}, [z^{(r)}|_{B_{r+1}}])
    &\supset  A_{r+1}^{(r+1)}\cup A_{r+2}^{(r+1)}\cup \dotsc\cup A_{t}^{(r+1)}\\
    &\supset A_{r+1}^{(r+1)}\cup A_{r+2}^{(r+1)}\cup \dotsc\cup A_{i}^{(r+1)}.
\end{align*}
By (\romannumeral2), 
$z|_{B_{t}}=z^{(t)}|_{B_{t}}$,
then we have $z(hg)=z^{(t)}(hg)=z^{(r)}(h)$ for $g\in A_{r+1}^{(r+1)}\cup A_{r+2}^{(r+1)}\cup \dotsc\cup A_{i}^{(r+1)}$ and $h\in B_{r+1}$.

In conclusion, for any $g\in A_{r+1}^{(r+1)}\cup A_{r+2}^{(r+1)}\cup \dotsc\cup A_{i}^{(r+1)}$ and for any $h\in B_{r+1}$, 
we have $T_g(z)(h)=z(hg)=
z^{(i)}(hg)=z^{(r)}(h)$, which implies that $g\in N(z, [z^{(r)}|_{B_{r+1}}])$.

\smallskip

Now we claim that $z$ is a $\calf$-recurrent point of $\Sigma_2$. To see this, let $R$
be a neighborhood of $z$ and pick $r\in\bbn$ such that $[z^{(r)}|_{B_{r+1}}]\subset R$. Thus we have
\[N\big(z,R\big)\supset N\big(z,[z^{(r)}|_{B_{r+1}}]\big)\supset 
\bigcup_{i=r+1}^{\infty} (\bigcup_{j=r+1}^{i}A_{j}^{(r+1)})=\bigcup_{i=r+1}^{\infty} A_{i}^{(r+1)}\]
where the second inclusion holds by $(\romannumeral3)$. By the construction of $\{A_n^{(r+1)}\}_{n=1}^\infty$, 
\[\bigcup_{i=r+1}^{\infty} A_{i}^{(r+1)}\in\calf.\]
So $z$ is a $\calf$-recurrent point of $\Sigma_2$.

By (\romannumeral2) $[1]=\{z\in \Sigma_2\colon z(e)=1\}$ is a neighborhood of $z$. 
We conclude the proof by showing that $N(z,[1])\subset F$.
Note that $N(z,[1])=\{g\in G: T_gz\in[1]\}=\{g\in G: z(g)=1\}$. By the construction $B_n\subset B_{n+1}$ for any $n\in\bbn$ and $\bigcup_{n=1}^\infty B_n\supset \bigcup_{n=1}^\infty G_n=G$. Thus for any $g\in N(z,[1])$, there exists $r\in\bbn$ such that $g\in B_{r+1}$, then 
by $(\romannumeral2)$ $z(g)=z^{(r)}(g)=1$, which implies that $g\in N(z^{r},[1])$. 
So $N(z,[1])\subset \bigcup_{r=1}^\infty N(z^{(r)},[1])$. By hypothesis $(1)$, for each $r\in\bbn$, $N(z^{(r)},[1])\subset F_1$ so $N(z,[1])\subset F_1\subset F$.
\end{proof}

\begin{rem}
In Section 3,  we showed that
$\fps$ and $\finf$ satisfy the properties (P1) and (P2). 
If $G$ is amenable and  $\{F_n\}$ is a  F{\o}lner sequence in $G$, $\fpud^{\{F_n\}}$ and $\fpubd$ also
satisfy the properties (P1) and (P2). 
So we can apply Theorem~\ref{thm:main2} to Furstenberg families $\fps$, $\finf$,
$\fpud^{\{F_n\}}$ and $\fpubd$. 
\end{rem}

\begin{defn}
Let $(X,G)$ be a $G$-system. 
A pair $(x_1,x_2)\in X\times X$ is said to be \emph{proximal} if $\inf_{g\in G}d(gx_1, gx_2)=0$, and \emph{distal} if it is not proximal. 
A point $x\in X$ is called \emph{distal} if for any $y\in \overline{Gx}$ with $y\neq x$,
$(x,y)$ is  distal.
\end{defn}

\begin{defn}
If for any $G$-system $(Y,G)$ and any recurrent point $y\in Y$, $(x,y)$ is recurrent in the product system $(X\times Y,G)$, 
then we say that $x$ is \emph{product recurrent}.
\end{defn}

\begin{defn}
Let $G$ be a countable infinite discrete group. 
A subset $F\subset G$ is called \emph{central} if there exists a  
$G$-system $(X,G)$, a point $x\in X$, an almost periodic point $y\in X$ and a neighborhood $U$ of $y$ such that $(x,y)$ is proximal and 
$N(x,U)\subset F$. Denote by $\fcen$ the collection of all central subsets of $G$.

A subset $A\subset G$ is called IP$^*$-set (resp.\@ central$^*$-set) if for any IP-subset (reps.\@ central subset) $F$ of $G$, $A\cap F\not=\emptyset$.
Denote by $\fip^*$ and $\fcen^*$ the collection of all IP$^*$-subsets and central$^*$-subset of $G$.
It is not hard to see that $\ft\subset \fcen \subset \fip$ 
and $\fip^*\subset\fcen^*\subset\fs$, see e.g.\@ \cite{HS12}.
\end{defn}

The following characterizations of distal points were proved by Furstenberg in \cite{F81} for topological dynamical systems and 
\cite{EEN01} for $G$-systems (see Corollaries 5.30 and 5.36 of \cite{EEN01}). 

\begin{thm}\label{thm:distal}
Let $(X,G)$ be a $G$-system and $x\in X$.
Then the following assertions are equivalent:
	\begin{enumerate}
		\item $x$ is a distal point;
		\item $x$ is an $\fip^*$-recurrent point;
		\item $x$ is an $\fcen^*$-recurrent point;
        \item $x$ is a product recurrent point.
	\end{enumerate}
\end{thm}

The notion of weak product recurrence was first introduced in \cite{HO08} by Haddad and Ott for topological dynamical systems.
Let $(X,G)$ be a $G$-system and $x\in X$.
If for any $G$-system $(Y,G)$ and any almost periodic point $y\in Y$, $(x,y)$ is recurrent in the product system $(X\times Y,G)$, 
then we say that $x$ is \emph{weak product recurrent}.

In \cite{AF94} Auslander and Furstenberg  asked whether weak product recurrent point is product recurrent.
It is answered by Haddad and Ott  in \cite{HO08} negatively for topological  dynamical systems. 
In \cite{DSY12}, Dong, Shao and Ye related product recurrence with disjointness, which was introduced by Furstenberg in his seminal paper \cite{F67}, and proved that if a  non-trivial transitive system is disjoint from any minimal system, then every transitive point is weak product recurrent but not minimal. 
Here we generalize this result to $G$-systems. 

\begin{defn}
Let $(X,G)$ and $(Y,G)$ be two $G$-systems. 
We say that a nonempty closed subset
$J\subset X\times Y$ is a \emph{joining} of $(X,G)$ and $(Y,G)$ if it is $G$-invariant and its projections onto the first and second coordinates are $X$ and $Y$ respectively.

If every joining is equal to $X\times Y$,
then we say that $(X,G)$ and $(Y,G)$ are \emph{disjoint}.
\end{defn}

In \cite{GTWZ21},  Glasner et al.\@ showed that for any infinite discrete group $G$, the Bernoulli shift is disjoint from any minimal system.
Recently, Xu and Ye \cite{XY22} gave a necessary and sufficient
condition for a transitive system $(X,G)$ to be disjoint from any minimal system when $G$ is a countable discrete group. 
In the following we show that any  transitive point in such a non-trivial transitive system is weak product recurrent but not product recurrent, which shows that Question~\ref{question:AF94-2} is also negative for $G$-systems.

In \cite[Theorem 4.3]{DSY12} the authors proved the following result for a topological dynamical system $(X,T)$, we generalize the result to $G$-systems. 

\begin{thm}\label{thm:weak-prod-rec-not-prod-rec}
Let $(X,G)$ be a non-trivial transitive system. 
If $(X,G)$ is disjoint from any minimal system,
then every transitive point $x\in X$ is weak product recurrent but not product recurrent.
\end{thm}
\begin{proof}
Let $x$ be a transitive point in $(X,G)$. 
First we show that $x$ is weak product recurrent. 
Given any almost periodic point $y$ in a $G$-system $(Y,G)$, 
we need to show that $(x,y)$
is recurrent.
Since $x$ is transitive,
$\overline{G(x,y)}$
is a joining of $X$ and $\overline{Gy}$. Since $(X,G)$ is disjoint from any minimal system, in particular $(X,G)$ and $(\overline{Gy},G)$ are disjoint, thus 
$\overline{G(x,y)}=X\times \overline{Gy}$.
Then for any neighborhood $U\times V$ of $(x,y)$ in $X\times Y$,
$G(x,y)\cap (U\times (V\cap\overline{Gy}))$ is an infinite set, i.e. $(x,y)$
is recurrent.

Now we show that $x$ is not product recurrent. Since $\calf^*_{cen}\subset \calf_s$, 
by Theorem~\ref{thm:distal}, it is sufficient to show that $x$ is not  almost periodic.
Assume on the contrary that $x$ is an almost periodic point.
Then $(X,G)$ is a minimal system.
By the assumption, $(X,G)$ is disjoint from itself. 
It is clear that $\{(z,z):z\in X\}$ is a joining of $(X,G)$ and $(X,G)$. Since 
$(X,G)$ is non-trivial, $\{(z,z):z\in X\}\neq X\times X$. This is a contradiction.
\end{proof}

In \cite{OZ13}, Oprocha and Zhang showed that the intersection of a dynamical syndetic set and a thick set contains a recurrent time set of a
piecewise syndetic recurrent point for topological dynamical systems.
In fact, a subset of $\bbn_0$ is the intersection of a dynamical syndetic set and a thick set if and only if it is central, see e.g.\@ \cite[Theorem 3.7]{HSY20}.
Using Theorem \ref{thm:main2}, we generalize Oprocha and Zhang's result to $G$-systems.

\begin{lem}\label{lem:central-lem}
Let $G$ be a countable infinite discrete group with identity $e$ and $F\subset G$.
If $F$ is a central set with $e\in F$, then 
there exists a $G$-system $(X,G)$, an $\fps$-recurrent point $x\in X$ and a neighborhood $U$ of $x$ such that $N(x,U)\subset F$.
\end{lem}

\begin{proof}
It is sufficient to show that $F$ satisfies Theorem \ref{thm:main2} (2) for the case of $\calf=\fps$. That is, there exists a decreasing sequence $\{F_n\}$ of subsets of $F$ in $\fps$ such that for any $n\in\bbn$ and $f\in F_n$ there exists 
$m\in\bbn$ such that $f F_m \subset F_n$.

Since $F$ is a central set, by the definition,
there exists a $G$-system $(X,G)$, a point $x\in X$, an almost periodic point $y\in X$ and a neighborhood $U$ of $y$ such that $(x,y)$ is proximal and 
$N(x,U)\subset F$.	
Since $U$ is a neighborhood of $y$, 
there exists $\epsilon>0$ such that $B(y,\epsilon)\subset U$.

For $n\in\bbn$, define $F_n:=N((x,y),B(y,\frac{\epsilon}{n})\times B(y,\frac{\epsilon}{n}))$. It is clear that $F_n\subset F$ and $F_{n+1}\subset F_n$ for $n\in\bbn$.  
Fix $n\in\bbn$ and we will show that $F_n\in\fps$.
Let $A:=N(y,B(y,\frac{\epsilon}{2n}))$ and $B:=\{g\in G: d(gx,gy) <\frac{\epsilon}{2n}\}$.
Since $y$ is an almost periodic point, $A$ is a syndetic set.
Since $(x,y)$ is proximal, 
$B$ is a thick set.
For any $g\in A\cap B$,
$d(gx,y)\leq d(gx,gy)+d(gy,y)<\frac{\epsilon}{n}$, then $gx\in B(y,\frac{\epsilon}{n})$. Thus 
$A\cap B\subset N((x,y),B(y,\frac{\epsilon}{n})\times B(y,\frac{\epsilon}{n}))=F_n$ and $F_n\in\fps$. 

Now fix $F_n$ and $f\in F_n$.
Note that $f(x,y)\in B(y,\frac{\epsilon}{n})\times B(y,\frac{\epsilon}{n})$ and
$y\in f^{-1}B(y,\frac{\epsilon}{n})$.
It is clear that $f^{-1}B(y,\frac{\epsilon}{n})$ is a neighborhood of $y$, thus
there exists
$m\in\bbn$ such that 
$B(y,\frac{\epsilon}{m})\subset f^{-1}B(y,\frac{\epsilon}{n})$.
Then we have \[
    fN\Big((x,y),B(y,\frac{\epsilon}{m})\times B(y,\frac{\epsilon}{m})\Big)\subset N\Big((x,y),B(y,\frac{\epsilon}{n})\times B(y,\frac{\epsilon}{n})\Big), 
\]
i.e.
$fF_m\subset F_n$.
\end{proof}

In \cite{DSY12}, Dong, Shao and Ye further studied product recurrent properties via Furstenberg families.
Let $\calf$ be a Furstenberg family and $(X,G)$ be a $G$-system.
We say that a point $x\in X$ is \emph{$\calf$-product recurrent} if for any given $\calf$-recurrent point $y$ in any $G$-system $(Y,G)$, $(x,y)$ is recurrent in the product system $(X\times Y, G)$. Dong, Shao and Ye \cite{DSY12} asked a question that if $x$ is $\calf_{ps}$-product recurrent, is $x$ necessarily a distal point? In \cite{OZ13} Oprocha and Zhang gave a positive answer on this question for topological dynamical systems. In the following result we will answer this question for $G$-systems.

\begin{thm}\label{thm:distal2}
Let $(X,G)$ be a $G$-system and $x\in X$. 
Then the following assertions are equivalent:
\begin{enumerate}
	\item $x$ is distal;
	\item  $x$ is $\fps$-product recurrent;
  \item for every $\fps$-recurrent point $y$ in the Bernoulli shift $(\Sigma_2,G)$, $(x,y)$ is recurrent in the product system $(X\times \Sigma_2,G)$.
\end{enumerate}
\end{thm}

\begin{proof}
    (1)$\Rightarrow$(2).
It follows from Theorem \ref{thm:distal}.

	(2)$\Rightarrow$(3). It is clear.
	
	(3)$\Rightarrow$(1).
By Theorem \ref{thm:distal}	it is sufficient to show that $x$ is an $\fcen^*$-recurrent point. 
For any neighborhood $U$ of $x$ and any central subset $A$ of $G$, 
by Lemma~\ref{lem:central-lem} there exists a $G$-system $(Y,G)$, 
an $\fps$-recurrent point $y\in Y$ and 
a neighborhood $V$ of $y$ such that $N(y,V)\subset A\cup\{e\}$. 
Then by Proposition \ref{prop:rec-sigma}, 
there exists an $\fps$-recurrent point $z\in \Sigma_2$ with $z\in [1]$
such that $N(z,[1])\subset A\cup \{e\}$.
By (3), $(x,z)$ is recurrent. Thus 
\[N(x,U)\cap N(z,[1])=N((x,z),U\times [1]) \]
is an infinite set.
Then we have $N(x,U)\cap A\neq\emptyset$, which implies that $N(x,U)\in \fcen^*$.
\end{proof}

\section{Return time sets  for \texorpdfstring{$G$}{G}-systems on compact Hausdorff spaces}\label{section5}

In this section, by virtue of  the algebraic properties of the Stone-\v{C}ech compactification $\beta G$ of $G$, we investigate  return time sets for general $G$-systems on compact Hausdorff spaces. 

First, we briefly introduce the concept of a compact right topological semigroup and its basic properties. 
By a \emph{compact right topological semigroup}, we mean a triple $(E,\cdot,\mathcal{T})$, where $(E,\cdot)$ is a semigroup, and $(E,\mathcal{T})$ is a compact Hausdorff space, and for every $p\in E$, the right translation $\rho_p\colon S\to S$, $q\mapsto q\cdot p$ is continuous. 
If there is no ambiguous, we will say that  $E$, instead of the triple $(E,\cdot,\mathcal{T})$, is a compact right topological semigroup.
A nonempty subset $L$ of $E$ is called a \emph{left ideal} of $E$ if  $E \cdot L\subset L$; is called a \emph{right ideal} of $E$ if  $L \cdot E\subset L$. 
A \emph{minimal left ideal} is the left ideal that does not contain any proper left ideal. 
A subset $I$ of $E$ is called an \emph{ideal} of $E$ if $I$ is both a left ideal and a right ideal of $E$. 
It is well known that $E$ has a smallest ideal, denoted by $K(E)$, which is the union of all minimal left ideals of $E$, see e.g.\@ \cite[Theorem 2.8]{HS12}. 
An element $p\in E$ is called \emph{idempotent} if $p\cdot p=p$.
An idempotent $p\in E$ is called a \emph{minimal idempotent} if there exists a minimal left ideal $L$ of $E$ such that $p\in L$.  
The following celebrated Ellis-Namakura Theorem reveals every compact right topological semigroup must contains an idempotent, see e.g.~\cite[Theorem 2.5]{HS12}.

\begin{thm}\label{thm:Ellis-Namakura}
Let $E$ be a compact right topological semigroup.
Then there exists $p\in E$ such that $p\cdot p=p$.
\end{thm} 

Now we recall the definition and algebraic structure of Stone-\v{C}ech compactification of a countable infinite discrete group. 
For further details on this topic, we refer the reader to the book \cite{HS12}. 
Let $G$ be a countable infinite discrete group and $\beta G$ be the collection of ultrafilters on $G$. By  Theorem 3.6 in \cite{HS12}, we know that each ultrafilter has the Ramsey property. 
Given $A\subset G$, let $\widehat{A}:=\{p\in\beta G: A\in p\}$.
If $g\in G$, then $\mathfrak{e}(g):=\{A\in \mathcal{P}(G): g\in A\}$ is easily seen to be an ultrafilter on $G$, which is called the \emph{principal ultrafilter} defined by $g$. 
Once we have identified $g\in G$ with $\mathfrak{e}(g)\in\beta G$, we shall suppose that $G\subset \beta G$.
In fact, the set $\{\widehat{A}: A\subset G \}$ forms a basis of a topology $\mathcal{T}$ on $\beta G$ (see\cite[Section 3.2]{HS12}). 
Then $(\beta G,\mathcal{T})$ is the Stone-\v{C}ech compactification of $G$ (see\cite[Section 3.3]{HS12}),
that is, for any compact Hausdorff space $Y$ and any function $\varphi\colon G\to Y$ there exists a continuous function $\widetilde\varphi\colon \beta G \to Y$ such that $\widetilde\varphi|_G=\varphi$. 
The operation $\cdot$ on $G$ can be uniquely extended to an operation $\cdot$ on $\beta G$ such that for any $p,q\in \beta G$, $p\cdot q=\{A\subset G:\{x\in G: x^{-1}A\in q\}\in p\}$.
Then $(\beta G,\cdot,\mathcal{T})$ is a compact Hausdorff right topological semigroup.

Recall that we introduced the definition of central set in Section 4. In \cite{BH90} Bergelson and Hindman obtained
the following characterization  of  central sets via the algebra properties of $\beta G$.

\begin{thm}\label{thm:central-set-idempotent}
Let $G$ be a countable infinite discrete group. 
A subset $F$ of $G$ is central if and only if there exists a minimal idempotent $p\in\beta G$ such that $F\in p$.
\end{thm}

The extension of the operation $\cdot$ on $G$ can be expressed by $p$-limits. 
We refer to \cite[Section 3.5]{HS12} for more about $p$-limits.

\begin{defn}
Let $p\in\beta G$, $\{x_g\}_{g\in G}$ be an indexed family in a compact Hausdorff space $X$ and $y\in X$. 
If for every neighborhood $U$ of $y$, $\{g\in G\colon x_g\in U\}\in p$,
then we say that the \emph{$p$-limit} of $\{x_g\}_{g\in G}$ is $y$, denoted by $p\text{-}\lim_{g\in G}x_g=y$. 
As $X$ is a compact Hausdorff space, $p\text{-}\lim_{g\in G}x_g$ exists and is unique. 

If viewing $\{g\}_{g\in G}$ as an indexed family in $\beta G$, then $p\text{-}\lim_{g\in G}g=p$.
\end{defn}

For a Furstenberg family $\calf\subset \calpg$, the \emph{hull} of $\calf$ is defined as  
\[h(\calf)=\{p\in\beta G: p\subset\calf\}.\]
If $\calf$ has the Ramsey property, then $h(\calf)$ is a nonempty closed subset of $\beta G$. 
For further details on this notion, we refer to  \cite{G80},  which in fact establishes a one-to-one correspondence between the set of Furstenberg families with the Ramsey property and the set of nonempty closed subsets of $\beta G$. 

A Furstenberg family $\calf\subset \calpg$ is called \emph{left shift-invariant} if for any $A\in \calf$ and $g\in G$, $gA\in\calf$.
We have the following equivalent condition for $h(\calf)$ to be a nonempty closed left ideal, see  \cite[Lemma 3.4]{L12} for the case $\bbn$ and \cite[Theorem 5.1.2]{C14} for a general discrete group.

\begin{lem}\label{lem:h-f}
Let $G$ be a countable infinite discrete group and $\calf\subset \calpg$ be a Furstenberg family with the Ramsey property.
Then $h(\calf)$ is a nonempty closed left ideal of $\beta G$ if and only if 
$\calf$ is left shift-invariant.
\end{lem}

The following lemma is folklore, see e.g. \cite[Theorem 4.4]{L12} or \cite[Lemma 5.2.2]{C14}.

\begin{lem} \label{lem:f-rec-idempotent}
Let $G$ be a countable infinite discrete group and $\calf\subset \calpg$ be a Furstenberg family with the Ramsey property.
If $h(\calf)$ is a nonempty closed subsemigroup of $\beta G$, then for any $G$-system $(X,G)$ on a compact Hausdorff space $X$, a point $x\in X$ is $\calf$-recurrent if and only if there exists an idempotent $p\in h(\calf)$ such that $p\text{-}\lim_{g\in G}gx=x$. 
\end{lem}

We say a subset $F$ of $G$ is an \emph{essential $\calf$-set}
if there exists an idempotent $p\in h(\calf)$ such that $F\in p$.
We present the following combinatorial characterization of essential $\calf$-sets, 
which was proved in \cite[Proposition 4.13]{L12} for the case of $\bbn$; however, it is routine to verify the proof extends to a general countably  infinite discrete group $G$. 

\begin{prop}\label{prop:ess-rec}
Let $G$ be a countable infinite discrete group and $\calf\subset \calpg$ be a Furstenberg family with the Ramsey property.
If $h(\calf)$ is a nonempty closed subsemigroup of $\beta G$,
then a subset $F$ of $G$ is an essential $\calf$-set if and only if there exists a decreasing sequence $\{F_n\}$ of subsets of $F$ in $\calf$ such that for any $n\in\bbn$ and $f\in F_n$ there exists $m\in\bbn$ such that $f F_m \subset F_n$.
\end{prop}

Now we have the following main result of this section, which characterizes the recurrent time sets of $\calf$-recurrent points in a $G$-system on a compact Hausdorff space.

\begin{thm}\label{thm:main3}
Let $G$ be a countable infinite discrete group with identity $e$ and $\calf\subset \calpg$ be a Furstenberg family with the Ramsey property.
If $\calf$ satisfies (P1) and (P2) and $h(\calf)$ is a nonempty closed subsemigroup of $\beta G$, then
\begin{enumerate}
    \item for any $G$-system $(X,G)$ on a compact Hausdorff space $X$, 
    if a point $x\in X$ is $\calf$-recurrent, then for
every neighborhood $U$ of $x$, $N(x,U)$ is an essential $\calf$-set;
    \item for any essential $\calf$-subset  $F$ of $G$, 
    there exists a $G$-system 
    $(X,G)$, an $\calf$-recurrent point
    $x\in X$ and a neighborhood $U$ of $x$
    such that $N(x,U)\subset F\cup \{e\}$.
\end{enumerate} 
\end{thm}

\begin{proof}
(1) Let $(X,G)$ be a $G$-system and $x\in X$ be an $\calf$-recurrent point. 
As $h(\calf)$ is a nonempty closed subsemigroup of $\beta G$, by Lemma~\ref{lem:f-rec-idempotent} there exists an idempotent $p\in h(\calf)$ such that $p\text{-}\lim_{g\in G}gx=x$. 
For every neighborhood $U$ of $x$, $N(x,U)=\{g\in G\colon gx\in U\}\in p$.
So $N(x,U)$ is an essential $\calf$-set.

(2) Let $F\subset G$ be an essential $\calf$-set.
As $h(\calf)$ is a nonempty closed subsemigroup of $\beta G$,
by Proposition~\ref{prop:ess-rec} there exists a decreasing sequence $\{F_n\}$ of subsets of $F$ in $\calf$ such that for any $n\in\bbn$ and $f\in F_n$ there exists $m\in\bbn$ such that $f F_m \subset F_n$.
As $\calf$ satisfies (P1) and (P2), by Theorem~\ref{thm:main2} there exists a $G$-system $(X,G)$, an $\calf$-recurrent point $x\in X$ and a neighborhood $U$ of $x$ such that $N(x,U)\subset F\cup \{e\}$.
\end{proof}

The following examples show that some Furstenberg families introduced in Section 3 satisfy the conditions of Theorem~\ref{thm:main3}.

\begin{exam}
Recall that $\finf$ is the collection of all infinite subsets of $G$.
It is easy to verify that $\finf$ satisfies the properties (P1) and (P2) and has the Ramsey property.
Note that $h(\finf)=\beta G\setminus G$. Then $h(\finf)$ is a closed ideal of $\beta G$. 
Therefore, all the conditions of Theorem~\ref{thm:main3} are satisfied for 
$\finf$. 
By \cite[Theorem 5.12]{HS12} a subset $F$ of $G$ is an essential $\finf$-set if and only if it is an IP-set.
It should be noticed that the IP-set defined in this paper must be an infinite subset of $G$.  
So Theorem~\ref{thm:main3} for the Furstenberg family $\finf$ characterizes the recurrent time sets of recurrent points via IP-sets.
\end{exam}

\begin{exam}
Recall that $\fps$ is the collection of all piecewise syndetic subsets of $G$.
Then $\fps$ has the Ramsey property and
by Lemma~\ref{lem:fps-P1-P2} $\fps$ satisfies  (P1) and (P2).
We know that $h(\fps)=\cl_{\beta G} K(\beta G)$, see e.g.\@ \cite[Corollary 4.41]{HS12},
and $\cl_{\beta G} K(\beta G)$ is a closed ideal of $\beta G$, see e.g.\@ \cite[Theoerem 4.44]{HS12}.
Therefore, all the conditions of Theorem~\ref{thm:main3} are satisfied for 
$\fps$.  
Following \cite{HMS96}, we say that a subset $A$ of $G$ is \emph{quasi-central} if there exists an idempotent $p\in \cl_{\beta G} {K(\beta G)}$ such that $A\in p$.
So Theorem~\ref{thm:main3} for the Furstenberg family $\fps$ characterizes the recurrent time sets of $\fps$-recurrent points via quasi-central sets,
which is similar to Theorem~\ref{thm:main1} in the introduction. 
\end{exam} 

\begin{exam}
Let $G$ be a countable infinite discrete amenable group and $\{F_n\}$ be a  F{\o}lner sequence in $G$.
Recall that $\fpud^{\{F_n\}}$ and $\fpubd$ are the collection of all subset of $G$ with positive upper density with respect to $\{F_n\}$ and
the collection of all subsets of $G$ with positive upper Banach density. 
By Lemma~\ref{lem:fpud-fpubd-P1-P2} $\fpud^{\{F_n\}}$ and $\fpubd$ satisfy (P1) and (P2). 
By Lemma~\ref{lem:h-f}, $h(\fpud^{\{F_n\}})$ and $h(\fpubd)$ are closed left ideals of $\beta G$.
Therefore, all the conditions of Theorem~\ref{thm:main3} are satisfied for 
$\fpud^{\{F_n\}}$ and $\fpubd$.

Following \cite{BD08}, we say that a subset $A$ of $G$ is a \emph{D-set} if there exists an idempotent $p\in h(\fpubd)$ such that $A\in p$.
So Theorem~\ref{thm:main3} for the Furstenberg family $\fpubd$ characterizes the recurrent time sets of $\fpubd$-recurrent points via D-sets.
\end{exam}

\section{\texorpdfstring{$\beta G$}{beta G}-actions and product recurrence}

In \cite{AF94} Auslander and Furstenberg initiated the study of the action of a compact right topological semigroup on a compact Hausdorff space.
In this section, we will focus on the $\beta G$-action and give a sufficient condition for the closed semigroups $S$ of $\beta G$ for which an $S$-product recurrent point is a distal point.

\begin{defn}\label{defn:beta-G-action}
Let $G$ be a countable infinite discrete group and $\beta G$ be the Stone-\v{C}ech compactification of $G$. 
By an \emph{action} of $\beta G$ on a compact  Hausdorff space $X$,
we mean a map $\Phi\colon \beta G \times X\to X$, $(p,x)\mapsto px$, such that 
$p(qx)=(pq)x$, for all $p,q\in \beta G$ and $x\in X$, and 
such that for each $x\in X$ the map $\Phi_x: \beta G\to X$, $p\mapsto px$, is continuous.  For convenience, we denote such an action of 
$\beta G$ on $X$ as $(X,\beta G)$.
It should be noticed that it is not assume that for each $p\in \beta G$,
the map $X\to X$, $x\mapsto px$, is continuous.

For two actions $(X,\beta G)$ and $(Y,\beta G)$, define a map
	$\Psi\colon \beta G \times (X\times Y)\to X\times Y$, $(p,(x,y))\mapsto (px,py)$, then it is an action on $X\times Y$, 
	we denote such an action of 
	$\beta G$ on $X\times Y$ as $(X\times Y,\beta G)$.
\end{defn}

\begin{rem}\label{rem:p-limit-x}
Let $(X,\beta G)$ be a $\beta G$-action. By the definition of $\beta G$-action, for each $x\in X$,  $\Phi_x: p\mapsto px$ is a continuous map from $\beta G$ to $X$. 
For every neighborhood $V$ of $px$, there exists some $A\in p$  such that $\Phi_x(\widehat{A})\subset V$.
Since $p\text{-}\lim_{g\in G}g=p$, $\{g\in G:g\in \widehat{A}\}\in p$.
Note that
$\{g\in G:g\in \widehat{A}\}\subset \{g\in G: gx\in V\}$, 
so we have 
$\{g\in G: gx\in V\}\in p$. 
By the uniqueness of $p$-limit, $p\text{-}\lim_{g\in G}gx=px$.
\end{rem}

\begin{rem}\label{rem:induced}
When $(X,G)$ is a $G$-system with $X$ being a compact Hausdorff space, 
there is a naturally induced action of $\beta G$ on $X$. 
For every $g\in G$, we view $g$ as a continuous map from $X$ to $X$.
Define $\theta: G\to X^X$ by  $\theta(g)=g$. 
As $\beta G$ is the Stone-\v{C}ech compactification of $G$,  
$\theta$ has a continuous extension  $\widetilde{\theta}\colon \beta G\to X^X$.
By the map $\widetilde{\theta}$, $\beta G$ actions on $X$. 
\end{rem}

Now we recall some basic dynamical concepts in the context of $\beta G$-actions. 

\begin{defn}
Let $(X,\beta G)$ be a $\beta G$-action. 
We say that  a pair $(x,y)$ of points in $X$ is \emph{proximal} if there exists some $p\in \beta G$ such that $px=py$.
If $(x,y)$ is not proximal, then $(x,y)$ is said to be \emph{distal}.
A point $x\in X$ is called \emph{distal} if for any $y\in \beta G x$ with $y\neq x$, $(x,y)$ is distal. 
\end{defn}

\begin{defn}
Let $(X,\beta G)$ be a $\beta G$-action.
We say that a point $x\in X$ is \emph{recurrent} if there exists some $p\in \beta G\setminus G$ such that $px=x$, and \emph{almost periodic} if there exists some minimal idempotent $p$ in $\beta G$ such that $px=x$.

\begin{rem}
It should be noticed that the notation $(X,\beta G)$ denotes the action of $\beta G$ on $X$ as defined in Definition \ref{defn:beta-G-action}. In general $(X,\beta G)$ is not a dynamical system since it is not assume that the map $\Phi\colon \beta G \times X\to X$ is continuous in Definition \ref{defn:beta-G-action}.
Here we define the notions "proximal", "distal", "recurrent" and "almost periodic" for $(X,\beta G)$. It is not hard to see that if the $\beta G$-action is induced by a $G$-system (see Remark \ref{rem:induced}) then the notions of "proximal", "distal", "recurrent" and "almost periodic" introduced here agree with the corresponding notions for $G$-systems. 
\end{rem}

Let $S$ be a nonempty closed subsemigroup of $\beta G\setminus G$.
A point $x\in X$ is said to be \emph{$S$-recurrent} if there exists some $p\in S$ such that $px=x$.

It is easy to see that a point $x$ is recurrent of $(X,\beta G)$ if and only if there exists an idempotent $p\in \beta G\setminus G$ such that $px=x$,
and a point
is almost periodic of $(X,\beta G)$ if and only if it is 
$L$-recurrent for some minimal left ideal $L$ of $\beta G$.
If $x\in X$ and $u$ is a minimal idempotent in $\beta G$, 
then $(x,ux)$ is proximal of $(X,\beta G)$ as $u(ux)=ux$.
It follows that a distal point of $(X,\beta G)$ is almost periodic of $(X,\beta G)$.
\end{defn}

In \cite{AF94} Auslander and Furstenberg generalized the characterization of distal points to general compact right topological semigroup actions.

\begin{thm}[{\cite[Theorem 1]{AF94}}]\label{thm:distal-idem} 
Let $(X,\beta G)$ be a $\beta G$-action and $x\in X$.
Then the following assertions are equivalent:
\begin{enumerate}
    \item $x$ is a distal point;
    \item for any almost periodic point $y\in X$, 
    $(x,y)$ is almost periodic in $(X\times X,\beta G)$;
    \item for any $\beta G$-action $(Y,\beta G)$ and any almost periodic point $y\in Y$, 
    $(x,y)$ is an almost periodic point in $(X\times Y,\beta G)$;
    \item for any idempotent $p\in \beta G$, $px=x$;
    \item for any minimal idempotent $p\in \beta G$, $px=x$;
    \item there is a minimal left ideal $L$ in $\beta G$ such that 
    for any idempotent $p$ in $L$, $px=x$.
\end{enumerate}
\end{thm}

\begin{defn}
Let $(X,\beta G)$ be a $\beta G$-action
and $S$ be a nonempty closed subsemigroup of $\beta G\setminus G$.
A point $x\in X$ is said to be \emph{$S$-product recurrent} if  
for any $\beta G$-action $(Y,\beta G)$ and any $S$-recurrent point $y\in Y$,
$(x,y)$ is an $S$-recurrent point in $(X\times Y,\beta G)$,
and \emph{weakly $S$-product recurrent} if  
for any $\beta G$-action $(Y,\beta G)$ and any $S$-recurrent point $y\in Y$,
$(x,y)$ is a recurrent point in $(X\times Y,\beta G)$.

By Theorem \ref{thm:distal-idem}, if $L$ is a minimal left ideal in $\beta G$, then $L$-product recurrence coincides with distality.
\end{defn}

In \cite{AF94}, Auslander and Furstenberg studied the general compact right
topological semigroup $E$ actions on a compact Hausdorff space $X$.
They introduced the cancellation semigroup condition and showed that if a nonempty closed subsemigroup $S\subset E$ satisfies the cancellation semigroup condition and contains a minimal left ideal of $E$, then 
$S$-product recurrence coincides with distality, see \cite[Corollary 4 and Theorem 4]{AF94}. This inspires Auslander and Furstenberg to proposal the Question~\ref{question:AF94-1}.

We obtain the following sufficient conditions on the closed subsemigroup $S$ of $\beta G$ for which $S$-product recurrence coincides with distality, which partly answers Question~\ref{question:AF94-1} for $\beta G$-actions.
Note that Theorem~\ref{thm:main-result-S-product-recurrent} is a direct  consequence of the following result.

\begin{thm} \label{thm:action-pro-distal}
Let $(X,\beta G)$ be a $\beta G$-action and $x\in X$. 
If $S$ be a nonempty closed subsemigroup of $\beta G\setminus G$ with $K(\beta G)\subset S$, then the following assertions are equivalent:
\begin{enumerate}
    \item $x$ is distal;
    \item $x$ is $S$-product recurrent; 
    \item $x$ is weakly $S$-product recurrent. 
\end{enumerate} 
\end{thm}
\begin{proof}
(1)$\Rightarrow$(2).
Assume that $x$ is a distal point. Given any $S$-recurrent point $y$ in any action $(Y,\beta G)$, there exists $p\in S$ such that $py=y$. Let
$L:=\{q\in S:qy=y\}$. Then $L$ is a nonempty closed subsemigroup of $\beta G$.  
By Ellis-Namakura Theorem (Theorem~\ref{thm:Ellis-Namakura}) there exists an idempotent $u\in L$. 
That is, there exists an idempotent $u\in S$ such that $uy=y$. 
Since $x$ is a distal point, by Theorem \ref{thm:distal-idem}, 
$ux=x$, and then $u(x,y)=(x,y)$.
and then $(x,y)$ is $S$-recurrent
in $(X\times Y,\beta G)$. 

(2)$\Rightarrow$(3). It is clear.

(3)$\Rightarrow$(1).
Assume on the contrary  that $x$ is not distal.
Then by Theorem~\ref{thm:distal-idem}, there exists a minimal idempotent $p\in\beta G$ such that $px\neq x$.
By Remark~\ref{rem:p-limit-x} and the Ramsey property of ultrafilter, there exists a neighborhood $U$ of $x$ such that $\{g\in G\colon gx\in X\setminus U\}\in p$. 
By Theorem~\ref{thm:central-set-idempotent}, $\{g\in G\colon gx\in X\setminus U\}$ is a central set.
Now by Lemma~\ref{lem:central-lem} and Proposition~\ref{prop:rec-sigma},
there exists an $\fps$-recurrent point $y$ with $y\in[1]$ in the Bernoulli shift $(\Sigma_2,G)$ such that 
$N(y,[1])\subset \{g\in G\colon gx\in X\setminus U\} \cup\{e\}$. 
Let $(\Sigma_2,\beta G)$ be the action of $\beta G$ on $\Sigma_2$ induced by $(\Sigma_2,G)$. 
Since $h(\fps)=\cl_{\beta G} K(\beta G)$, by Lemma~\ref{lem:f-rec-idempotent}, Remark~\ref{rem:p-limit-x} and $\cl_{\beta S}{K(\beta G)}\subset S$, $y$ is $S$-recurrent in $(\Sigma_2,\beta G)$. 
As $x$ is weakly $S$-product recurrent, $(x,y)$ is recurrent in $(X\times \Sigma_2,\beta G)$.
But $\{g\in G\colon (gx,gy)\in U\times [1]\}\subset \{g\in G\colon gx\in U\}\cap (\{g\in G\colon gx\in X\setminus U\}\cup\{e\})=\{e\}$, which is a contradiction.
\end{proof}

Applying Theorem~\ref{thm:action-pro-distal}, we prove  Theorem~\ref{thm:main-result-distal-product-rec} as follows.

\begin{proof}[Proof of Theorem~\ref{thm:main-result-distal-product-rec}]
For a Furstenberg family $\calf\subset \calpg$,
if $\calf$ has the Ramsey property, then the hull $h(\calf)$ of $\calf$ is a nonempty closed subset of $\beta G\setminus G$.
If $\calf\supset \fps$, then $h(\calf)\supset h(\fps)=\cl_{\beta G} K(\beta G)$.
Let $(X,G)$ be a $G$-system. Consider the action $\beta G$ of $G$ induced by $(X,G)$.  By Lemma~\ref{lem:f-rec-idempotent} and Remark~\ref{rem:p-limit-x}, the result is an immediate  consequence of Theorem~\ref{thm:action-pro-distal}. 
\end{proof}

\begin{rem}
It should be noticed that Theorem~\ref{thm:main-result-distal-product-rec} holds for the Furstenberg families $\fps$ and $\finf$, and if in addition $G$ is amenable, then it holds for the Furstenberg family $\fpubd$.
\end{rem}

Let $G$ be a countable infinite discrete amenable group and $\{F_n\}$ be a F{\o}lner sequence in $G$. 
Recall that $\fpud^{\{F_n\}}$ is the collection of all subsets of $G$ with positive upper density with respect to $\{F_n\}$.
We know that $h(\fpud^{\{F_n\}})$ is a nonempty closed left ideal of $\beta G$. 
As $\fps\not\subset \fpud^{\{F_n\}}$, we can not apply Theorem~\ref{thm:main-result-distal-product-rec}. 
So we have the following natural question:
\begin{question} 
Is $\fpud^{\{F_n\}}$-product recurrence equivalent to distality? 
\end{question}

\noindent \textbf{Acknowledgment.}
J. Li was partially supported by National Key R\&D Program of China (No.~2024YFA1013601) and  NSF of China (Grant nos.~12222110 and 12171298). Y. Yang was partially supported by STU Scientific Research Initiation Grant (SRIG, No. NTF24025T).

\bibliographystyle{amsplain}

\end{document}